\documentclass[11pt, A4]{article}
\usepackage{geometry}
\usepackage{subcaption}
\usepackage{amsthm,amsmath,amssymb,amsfonts,hyperref}
\usepackage{epsfig,color}
\usepackage{enumerate}

\newtheorem{defn}{Definition}[section]
\newtheorem{thm}{Theorem}[section]
\newtheorem{prop}{Proposition}[section]

\newtheorem{rmk}{Remark}[section]
\newtheorem{lma}{Lemma}[section]

\newtheorem{exm}{Example}[section]

\newcommand{\eq}{\begin{equation}}
\newcommand{\qe}{\end{equation}}

\def\N{{\rm I\kern-0.16em N}}
\def\R{{\rm I\kern-0.16em R}}
\def\E{\mathbb{E}}
\def\P{{\rm I\kern-0.16em P}}
\def\F{{\rm I\kern-0.16em F}}
\def\B{{\rm I\kern-0.16em B}}
\def\C{{\rm I\kern-0.46em C}}
\def\G{{\rm I\kern-0.50em G}}
\def\Z{{\mathbb{Z}}}

\numberwithin{equation}{section}

\def\red#1{\textcolor{blue}{#1}}

\begin{document}

\title{Stein-type covariance identities:\\
  \large Klaassen, Papathanasiou and Olkin-Shepp--type bounds for
  arbitrary target distributions }

\date{December 21, 2018}
\renewcommand{\thefootnote}{\fnsymbol{footnote}}
\begin{small}
  \author{Marie Ernst\footnotemark[1], \, 
Gesine Reinert\footnotemark[2] \, and Yvik Swan\footnotemark[1]}
\end{small}
\footnotetext[1]{Universit\'e de Li\`ege, corresponding author Yvik
  Swan: yswan@uliege.be.}
\footnotetext[2]{University of Oxford.}

\maketitle

\begin{abstract}
  In this paper, following on from
  \cite{ley2015distances,LRS16, MRS18} we
  present a minimal formalism 
  for Stein operators which leads to different probabilistic
  representations of solutions to Stein equations. These in turn
  provide a wide family of Stein-Covariance identities which we put to
  use for revisiting the very classical topic of bounding the variance
  of functionals of random variables.  Applying the Cauchy-Schwarz
  inequality yields first order upper and lower
  Klaassen~\cite{K85}-type variance bounds. A probabilistic
  representation of Lagrange's identity (i.e.\ Cauchy-Schwarz with
  remainder) leads to
  Papathanasiou~\cite{papathanasiou1988variance}-type variance
  expansions of arbitrary order. A matrix Cauchy-Schwarz inequality
  leads to Olkin-Shepp~\cite{olkin2005matrix} type covariance bounds.  All results hold for
  univariate target distribution under very weak assumptions (in
  particular they hold for continuous and discrete distributions
  alike). Many concrete illustrations are provided.
\end{abstract}


\section{Introduction}
\label{sec:introduction}

Charles Stein's mathematical legacy is growing at a remarkable pace
and many of the techniques and concepts he pioneered are now a staple
of contemporary probability theory. The origins of this stream of
research lie in two papers: \cite{S72}, in which the method was first
presented in the context of Gaussian approximation, and \cite{C75}
where the method was first adapted to a non-Gaussian context, namely
that of Poisson approximation.  As has been noted by many authors
since then, the approach can be applied quasi verbatim to any target
distribution other than the Gaussian and the Poisson, under the
condition that ``correct'' \emph{ad hoc} objects be identified which
will permit the basic identities to hold.  There now exist several
excellent books and reviews on Stein's method and its consequences in
various settings, such as \cite{Stein1986,BC05,BC05b,NP11,ChGoSh11}.
There also exist several non-equivalent general frameworks for the
theory covering to large swaths of probability distributions, of which
we single out the works \cite{Do14,upadhye2014stein} for univariate
distributions under analytical assumptions, \cite{arras2017stein,
  arras2018stein} for infinitely divisible distributions and
\cite{mackey2016multivariate,gorham2016measuring,gorham2017measuring}
as well as \cite{fang2018multivariate} for multivariate densities
under diffusive assumptions. A ``canonical'' differential Stein
operator theory is also presented in
\cite{ley2015distances,LRS16,MRS18}.  

Stein's method can be broken down into a small number of key steps:
[A] identification of a (characterizing) linear operator, [B] bounding of
solutions to some differential equations related to this operator,
[C] probabilistic Taylor expansions and construction of well-designed
couplings; see \cite{Re04} for an overview.  Each of these steps
has produced an entire ecosystem of ``Stein-type objects'' (operators,
equations, couplings, etc.). These Stein-type objects 
are in symbiosis
with many classical branches of mathematics such as orthogonal
polynomials, functional analysis, PDE theory or Markov chain theory
and therefore open bridges between Stein's theory and these important
areas of mathematics. 
 More recently, connections with other more contemporary
mathematics have been discovered, such as e.g.\ information theory as
in \cite{nourdin2013entropy, arras2017formulae}, optimal
transportation as in \cite{ledoux2015stein, fathi2018stein}, and machine
learning as in
\cite{gorham2015measuring,liu2016kernelized,chwialkowski2016kernel}.

In the present paper, we pursue the work begun in
\cite{ley2015distances,LRS16,MRS18} and adopt a minimal point of view
on all the objects concerned, this time concentrating on the solutions
to so-called ``Stein equations''. Aside from its intrinsic interest --
which may arguably be only of concern to those meddling directly in
the method itself -- we shall illustrate the power of our formalism by
showing how it allows to obtain optimal and (extremely) flexible upper
and lower bounds on arbitrary functionals of random variables with
arbitrary univariate distribution. We recover, as particular cases,
many (all?)  previously known ``variance bounds'' (also called
Poincar\'e inequalities), including Klaassen's bounds from \cite{K85},
Papathanasiou's variance expansions from
\cite{papathanasiou1988variance} (and Houdr\'e and Kagan's famous
expansion from \cite{HK95}) as well as Olkin and Shepp's bound from
\cite{olkin2005matrix}. Our method of proof is in each case new and,
moreover, the formalism we introduce makes them in some sense
elementary -- at least as soon as the framework is laid out. In order
to ease the reader into our work, we begin by detailing it in the
easiest setting, namely that of a Gaussian target distribution.

\subsection{The Gaussian case}
\label{sec:gaussian-case}

The Gaussian density  $ \gamma(x) = (2\pi\sigma^2)^{-1/2}
\mathrm{exp}(-(x-\mu)^2/2\sigma^2)$ has many remarkable
properties. One of them stems from Stein's celebrated lemma which
reads: a random variable $N$ has distribution $N \sim \mathcal{N}(\mu, \sigma^2)$ if
and only if 
  \begin{equation}
    \label{eq:stlma}
    \mathbb{E} \left[ (N-\mu) g(N) \right] = \mathbb{E} \left[ \sigma^2
      g'(N)\right]
  \end{equation}
  $\mbox{for all }g:\R \to \R \mbox{ such that }
  \mathbb{E}[|g'(N)|]<\infty. $ We let $\mathcal{F}(\gamma)$ be the
  collection of $g$ such that $\mathbb{E}|g'(N)|<\infty$. There are
  many ways to prove \eqref{eq:stlma}, but the two basic approaches are:
\begin{itemize}
\item  Approach 1:  use the fact that the Gaussian \emph{score function} $\rho(x)= (\log \gamma(x))' = \gamma'(x)/\gamma(x) =
  -(x-\mu)/\sigma^2$  is linear, in combination with integration by parts (see \cite{Stein2004});
\item Approach 2: use the fact that the Gaussian \emph{Stein kernel}
  $\tau(x) = \gamma(x)^{-1} \int_{-\infty}^x(u-\mu) \gamma(u) du =
  \sigma^2$ is constant, in combination with Fubini's theorem (see
  \cite[Lemma 1.2]{NP11}).
\end{itemize}
Contrarily to appearances, the final result obtained
via these two approaches is not
identical because (i)~there are technical differences concerning the
classes of test functions to which the resulting identities apply;
(ii)~they lead to two formally different identities: where the
variance is on the right-hand-side of the equality \eqref{eq:stlma}
via Approach 2, it is at the denominator of the left hand side of the
equality via Approach 1. This is not a simple cosmetic difference that
can be brushed away as a byproduct of the standardization, it is
rather central to the understanding of the very nature of Stein's
operators and the innocuity of the difference is rather characteristic
of the Gaussian distribution. In the language of the present paper,
many of the remarkable properties of the Gaussian actually stem from a
very ``Steinian'' characteristic property of the Gaussian: it is the
only distribution whose Stein kernel
$\tau(x) = \gamma(x)^{-1} \int_{-\infty}^x(u-\mu) \gamma(u) du$ is
constant.

In order to exploit Stein's identity \eqref{eq:stlma} in the context
of the so-called \emph{Stein's method}, one starts by considering
solutions to the so-called \emph{Stein equations}
\begin{equation}\label{eq:68}
  \sigma^2f_h'(x) - (x-\mu) f_h(x) = h(x) - \mathbb{E}[h(N)]
\end{equation}
where $h \in \mathcal{H}$ some class of test functions. For any given
$h \in L^1(\gamma)$ there exists (a.s.) a unique solution $ \mathcal{L}_\gamma h=f_h$  to
\eqref{eq:68} such that $ f_h$ is bounded on $\R$, given by
\begin{equation}
  \label{eq:58}
 \mathcal{L}_{\gamma} h (x)= f_h(x)= \frac{1}{\gamma(x)}
  \int_{-\infty}^x \big(h(u) - \mathbb{E}[h(N)]\big)\gamma(u) du = -\frac{1}{\gamma(x)}
  \int_x^{\infty} \big(h(u) - \mathbb{E}[h(N)]\big)\gamma(u) du.
\end{equation}
The operator $h \mapsto \mathcal{L}_\gamma h$ is called the
\emph{pseudo inverse Stein operator}.
Such functions as \eqref{eq:58} can be used in
order to assess normality of a real valued random variable $X$ through
the identities:
\begin{align}
  \label{eq:48}
  \sup_{h{\in}\mathcal H} \mathbb{E}[h(X)] - \mathbb{E}[h(N)] & =
                                                               \sup_{h\in
                                                               \mathcal{H}}
                                                               \mathbb{E}
                                                               [
                                                               \sigma^2
                                                               f_h'(X) -
                                                               (X-\mu)
                                                               f_h(X) ]
  \\
  & \le \sup_{f \in \mathcal{F}} \mathbb{E}
                                                               [
                                                               \sigma^2
                                                               f'(X) -
                                                               (X-\mu)
                                                               f(X) ] \label{eq:59}
\end{align}
where
$\mathcal{F}\supset \left\{ \mathcal{L}_\gamma h\, | \, h \in
  \mathcal{H}\right\}$ is to be ``well-chosen''. The expressions on
the left hand side of \eqref{eq:48} are classical, as they correspond
to the so-called Integral Probability Metric (IPM): Wasserstein-1
distance for $\mathcal{H}_{\mathrm{Lip}}=$Lip(1) the class of all
Lipschitz functions with constant 1; Total Variation distance for
$\mathcal{H}_{TV}= \left\{ \mathbb{I}_A, A \subset \R \right\}$ the
class of all indicators of all Borel sets in $\R$; Kolmogorov distance
for
$\mathcal{H}_{\mathrm{Kol}}=\left\{ \mathbb{I}_{(-\infty, z]}, z \in
  \R \right\}$ the class of indicators of half lines. These are
natural measures of probabilistic discrepancy.  The right hand side of
\eqref{eq:59} is more mysterious, and one of the secrets of its
usefulness lies in the fact that one can chose the class $\mathcal{F}$
to be of a very simple nature.  Indeed, we start with the observation
that there exist constants $\kappa^i_{\mathcal{H}}, i=1, 2, \ldots$ such
that the functions $f_h$ satisfy the uniform bounds 
\begin{align}
  \label{eq:64}
 \sup_{h \in \mathcal{H}} \| f_h \|_{\infty}  \le
  \kappa^1_{\mathcal{H}}, \quad  \sup_{h \in \mathcal{H}} \|
f_h' \|_{\infty}  \le \kappa_{\mathcal{H}}^2, \quad \sup_{h \in \mathcal{H}} \|
f_h'' \|_{\infty}  \le \kappa_{\mathcal{H}}^3 
  \quad  \ldots 
\end{align}
for all important classes $\mathcal{H}$ (including the three mentioned
above) -- these are the so-called ``Stein's factors''. Given bounds
such as \eqref{eq:64} one can take
\begin{align}
  \label{eq:65}
  \mathcal{F} = \left\{ f : \R \to R \mbox{ such that } \|f\|_{\infty} \le
  \kappa_{\mathcal{H}}^1, \|f'\|_{\infty} \le \kappa_{\mathcal{H}}^2, \mbox{ etc} \right\}
\end{align}
so that
$\mathcal{F}\supset \left\{ f_h \mbox{ of the form \eqref{eq:58}}
\right\}$ and, crucially, the class $\mathcal{F}$ has a \emph{simple}
structure. This makes \eqref{eq:59} a very potent starting point for
assessing normality of $X$.

Starting from \eqref{eq:stlma}, it is natural to consider the
\emph{Stein operator}
$\mathcal{A}_{\gamma}f(x) =\sigma^2 f'(x) - (x-\mu) f(x)$ which has
the property that $X = N$ (equality in distribution) if and only if
$\mathbb{E}[ \mathcal{A}_{\gamma}f(N)] = 0$ for all
$f \in \mathcal{F}(\gamma)$.  In light of the arguments from the
previous paragraph, the inverse of
$\mathcal{A}_{\gamma}$ is given by the 
pseudo inverse Stein operator which is the integral
operator which to any $h \in L^1(\gamma)$ associates the function $ \mathcal{L}_{\gamma} h$ given in \eqref{eq:58}. 
With this construction, 
$\mathcal{A}_{\gamma} \mathcal{L}_{\gamma} h = h-\mathbb{E}[h[N]]$ for
all $h \in L^1(\gamma)$ and
$ \mathcal{L}_{\gamma} \mathcal{A}_{\gamma} h = h$ for all
$h \in \mathcal{F}(\gamma)$.  

Here are some important examples. 
With $g_z(x) = \mathbb{I}(x \le z),$ 
\begin{equation} \label{eq:g1}
 \mathcal{L}_{\gamma} g_z(x) = \frac{1}{\gamma(x)} \left( \mathbb{P} (N \le x \wedge z) -  \mathbb{P} (N \le x)  \mathbb{P} (N \le  z) \right) . 
\end{equation}
As a second example, the Stein kernel $\tau(x)$ is nothing
other than the solution \eqref{eq:58} evaluated at $h = \mathrm{Id}$
the identity function;  $ \mathcal{L}_{\gamma}\mathrm{Id}(x) = \tau(x)$
which, as already mentioned, is constant and equal to $\sigma^2$ at
all $x$.
From Stein's identity we deduce that the Gaussian density is the
only density for which the function $ \mathcal{L}_{\gamma}\mathrm{Id}$
is constant. {{The operator $\mathcal{L}_\gamma$ evaluated at the
    identity function in general gives the zero bias density from
    \cite{GR97}, in the sense that
    $ \mathbb{E} [X h (X)] = \mathbb{E} [ \{
    -\mathcal{L}_{\gamma}\mathrm{Id}(X) \} h'(X)] = \mathbb{E} [
    h'(X^*)]$, with $X^*$ having the $X$-zero bias density.  }}

The Stein operator definition  gives in particular that, for $f \in \mathcal{F}(\gamma)$ and $ h \in   L^1(\gamma)$,
 $$\mathbb{E} \left[ \{ \mathcal{A}_{\gamma}f(N) \} h(N) \right] =
 \sigma^2 \mathbb{E} [ f'(N) h(N) - (fh)'(N)] = 
   -\mathbb{E}[f(N) h'(N)].$$
  Letting $g(x) =\mathcal{A}_{\gamma}f(x) $ and using  $ \mathcal{L}_{\gamma} \mathcal{A}_{\gamma} h = h$ allows to generalize \eqref{eq:stlma}
to the   \emph{covariance identity}:
\begin{align}
  \label{eq:stekovidentity}
\mathrm{Cov}\left[  g(N), h(N) \right] = \mathbb{E} \left[
  \big(-\mathcal{L}_{\gamma}g(N)\big)\,h'(N)  \right] 
\end{align}
which is valid for all $h, g \in L^1(\gamma)$. 
For further covariance inequalities we start
with the (trivial) observation that
\begin{align}
  \label{eq:63}
  \mathcal{L}_{\gamma}h(x) = \frac{1}{\gamma(x)}\mathbb{E} \left[ \big(h(N)
  -\mathbb{E}[h(N)]\big) \mathbb{I}[N \le x] \right]=
  \mbox{Cov}\bigg[h(N), 
  \frac{\mathbb{I}[N\le x]}{\gamma(x)}\bigg].
\end{align} 
Then the following holds. 
\begin{lma}[Representation of the inverse Stein
  operator]\label{lma:reprgaussiancase}
  Let $N_1, N_2$ be independent copies of $N$. Then 
  \begin{equation}
    \label{eq:62}
 - \mathcal{L}_{\gamma}h(x) = \frac{1}{\gamma(x)}\mathbb{E} \left[ \big(h(N_2)-
h(N_1)\big) \mathbb{I}[N_1 \le x \le N_2]\right]
\end{equation}
for all $h \in L^1(\gamma)$.
\end{lma}
The proof of Lemma \ref{lma:reprgaussiancase} follows simply by
expanding \eqref{eq:62} and showing that it is equal to \eqref{eq:63}
for all $x$. We will provide details in a (much) more general context
in Section \ref{sec:variance-expansion} (see Lemma \ref{lem:represnt}).
    
 Identity \eqref{eq:62} is not the only available probabilistic
representation for $\mathcal{L}_\gamma$.  The next one we found
in \cite[Proposition 1]{saumard2018weighted}.

\begin{lma}[Saumard's lemma] The symmetric kernel 
  \begin{equation}\label{eq:66}
    K_\gamma(x, x') = \mathbb{P}[N \le x\wedge x'] - \mathbb{P}[N\le x]
    \mathbb{P}[N \le x']
  \end{equation}
  is  positive definite. Moreover 
  \begin{equation}
    \label{eq:67}
    -\mathcal{L}_{\gamma}h(x) = \mathbb{E} \left[ h'(N)
      \frac{K_{\gamma}(N, x)}{\gamma(N) \gamma(x)}\right]  
  \end{equation}
  for all absolutely continuous $h \in L^1(\gamma)$, and letting $N_1,
  N_2$ be independent copies of $N$ we have 
  \begin{equation}
    \label{eq:11}
  \mathrm{Cov}[h(N), g(N)] = \mathbb{E} \left[ h'(N_1)
    \frac{K_{\gamma}(N_1, N_2)}{\gamma(N_1) \gamma(N_2)} g'(N_2)  \right]
\end{equation}
for all absolutely continuous  $h, g \in L^2(\gamma)$. 
\end{lma}

\begin{rmk}
  Functions $x' \mapsto K_\gamma(x, x')/\gamma(x')$ are represented in
  Figure \ref{fig:sfig11} for various values of $x$.
\end{rmk}

\begin{proof}
Symmetry of  \eqref{eq:66} is immediate. To see that it is   positive,
we note that 
  \begin{align*}
    K_{\gamma}(x, x') =
    \begin{cases}
      \mathbb{P}[N \le x]\mathbb{P}[N\ge x'] & \mbox{ if } x \le x'\\
      \mathbb{P}[N \le x']\mathbb{P}[N\ge x] & \mbox{ if } x >  x'
    \end{cases}.
  \end{align*}
  To obtain \eqref{eq:67}, we use \eqref{eq:stekovidentity} inside
  \eqref{eq:63}: recalling the notation $g_x(n) = \mathbb{I}[n \le x]$ we have 
\begin{align*}
  -\mathcal{L}_{\gamma} h(x) & = \frac{1}{\gamma(x)} \mathrm{Cov} \left[
h(N),  g_x(N) \right] 
   = \frac{1}{\gamma(x)} \mathbb{E} \left[
h'(N) \big(-\mathcal{L}_{\gamma}g_x(N)) \right].  
\end{align*}
{{
Using \eqref{eq:g1}, with $N'$ an independent copy of $N$, 
$$
-\mathcal{L}_{\gamma}g_x(N)= \frac{1}{\gamma(N)} \left( \mathbb{P} (N' \le x \wedge N) -  \mathbb{P} (N' \le x)  \mathbb{P} (N' \le  N) \right) . 
$$ Hence
\begin{align*}
  -\mathcal{L}_{\gamma} h(x) & = \mathbb{E} \left[ \frac{1}{\gamma(x)\gamma(N)} \left( \mathbb{P} (N' \le x \wedge N) -  \mathbb{P} (N' \le x)  \mathbb{P} (N' \le  N) \right) 
h'(N)\right] 
\end{align*}
and \eqref{eq:67} follows. 
}}
For the last point, we simply combine \eqref{eq:stekovidentity} and
\eqref{eq:67} again, with $N_2$ yet another independent copy of
$N$. Then
\begin{align*}
  \mathrm{Cov}[h(N), g(N)] & = \mathbb{E} \left[ h'(N_1)
                            \big(-  \mathcal{L}_{\gamma}g(N_1)\big) \right] 
    = \mathbb{E} \left[ h'(N_1) \mathbb{E}\bigg[
    g'(N_2)\frac{K_{\gamma}(N_1, N_2)}{\gamma(N_1) \gamma(N_2)}  \, | \,
    N_1\bigg] \right]
\end{align*}
and the conclusion follows. 
\end{proof}
We conclude this introduction by showing how the notations we have
introduced are not only of cosmetic value, but that they allow to
obtain some powerful results in an efficient manner. For instance,
starting from Saumard's lemma combined with a simple application of
the Cauchy-Schwarz inequality, for any test function such that $h'>0$
a.s.\ then
\begin{align*}
  \mathrm{Var}[g(N)] & =  \mathbb{E}\left[ g'(N_1)
                       \frac{K_{\gamma}(N_1, N_2)}{\gamma(N_1)
                       \gamma(N_2)} g'(N_2) \right] \\
  & =  \mathbb{E}\left[ \frac{ g'(N_1)}{\sqrt{h'(N_1)}}
                       \sqrt{\frac{K_{\gamma}(N_1, N_2)h'(N_2)}{\gamma(N_1)
                       \gamma(N_2)}}  \sqrt{\frac{K_{\gamma}(N_1,
    N_2)h'(N_1)}{\gamma(N_1) 
    \gamma(N_2)}} \frac{g'(N_2)}{\sqrt{h'(N_2)}} \right]
  \\
  & \le \mathbb{E}\left[  \frac{ g'(N_1))^2}{h'(N_1)} 
\frac{K_{\gamma}(N_1,
    N_2)h'(N_2)}{\gamma(N_1) 
    \gamma(N_2)}\right] \\
                     & = \mathbb{E}\left[ \frac{ g'(N_1))^2}{h'(N_1)}
                   \mathbb{E} \left[ \frac{K_{\gamma}(N_1,
    N_2)h'(N_2)}{\gamma(N_1) 
    \gamma(N_2)} \, \bigg|\, N_1\right] \right] = \mathbb{E}\left[
                       \frac{ g'(N_1))^2}{h'(N_1)} 
 \big(-\mathcal{L}_{\gamma}h(N_1)   \big) \right] 
\end{align*}
where we applied \eqref{eq:67} once again.  Lower bounds are just
as easy to obtain, from \eqref{eq:stekovidentity}:
\begin{align*}
\mathbb{E} \left[ \big(
  -\mathcal{L}_{\gamma}h(N)\big) g'(N)\right]^2 \le
  \mathbb{E}[h(N)^2] \mathrm{Var}[g(N)]
\end{align*}
which leads to the fact that, 
if $h \in L^1(\gamma)$ is monotone,  then 
\begin{align}
  \label{eq:69}
  \frac{ \mathbb{E} \left[ \big(
  -\mathcal{L}_{\gamma}h(N)\big) g'(N)^2\right]}{
  \mathbb{E}[h(N)^2]} \le \mathrm{Var}[g(N)] \le  \mathbb{E}\left[ \left( \frac{ g'(N_1))^2}{h'(N_1)} \right)^2
 \big(-\mathcal{L}_{\gamma}h(N_1) \big)  \right].
\end{align}
In particular, taking $h = \mathrm{Id} $
\begin{equation}
  \label{eq:70}
  \sigma^2 \mathbb{E}[g'(N)]^2 \le \mbox{Var}[g(X)] \le
  \sigma^2\mathbb{E}[g'(X)^2]. 
\end{equation}
Identity \eqref{eq:70} is a rephrasing of Chernoff \cite{C81}'s
classical Gaussian bounds; identity \eqref{eq:69} is Klaassen's result
\cite{K85} in the Gaussian case. As we shall see in Section
\ref{sec:variance-expansion}, one can push the argumentation much
further and obtain upper and lower infinite variance expansions to any
order. In particular we recover the famous variance bound from
\cite{HK95} (which was already available in
\cite{papathanasiou1988variance}): {{for all $n \ge 1$,}}
\begin{equation}
  \label{eq:71}
  \sum_{j=1}^{2n} (-1)^j \frac{(\sigma^2)^j}{j!} \mathbb{E} [(g^{(j)}(N))^2]
  \le \mbox{Var}[g(N)] \le   \sum_{j=1}^{2n+1} (-1)^j
  \frac{(\sigma^2)^j}{j!} \mathbb{E} [g^{(j)}(N)^2]. 
\end{equation}
Moreover, 
we will prove in Section
\ref{sec:mult-extens-olkin} the \emph{matrix variance bound}:
\begin{align}\nonumber
  &   \begin{pmatrix}
      \mathrm{Var}[f(X)] & \mathrm{Cov}[f(X), g(X)] \\
      \mathrm{Cov}[f(X), g(X)] & \mathrm{Var}[g(X)]
    \end{pmatrix} 
    \le \label{eq:77}
\sigma^2 
 \begin{pmatrix}\mathbb{E} \left[\big(f'(X)\big)^2 \right]&
\mathbb{E} \left[f'(X)g'(X)\right]\\
\mathbb{E} \left[f'(X)g'(X)\right] & \mathbb{E}
  \left[\big(g'(X)\big)^2 \right]
\end{pmatrix} 
  \end{align}
(inequality in the above indicating that the difference is non
negative definite), hereby recovering the main result of
\cite{olkin2005matrix}.   Again, our approach applies
 to
basically \emph{any} univariate target distribution. 
\subsection{Some references}
\label{sec:variance-bounds}

Extensions of Chernoff's first order bound \eqref{eq:70} have, of
course, attracted much attention. The initiators of the stream of
research seem to be \cite{ch82} and \cite{C82}, although precursors
can be found e.g.\ in \cite{BrLi76}. Chen \cite{ch82} identified a way
to exploit Stein's operator for the Gaussian distribution not only to
simplify Chernoff's proof,
but also to propose a first order upper variance bound for the
multivariate Gaussian distribution.  \cite{C82} identifies the role
played by the Stein kernel (and its discrete version) to extend the
scope of Chernoff and Chen's bounds to a very wide class of
distributions. A remarkable generalization -- and one of the main
sources of inspiration behind the current article -- is \cite{K85} who
pinpoints the role plaid by Stein inverse operators in such bounds,
and obtains the inequalities in \eqref{eq:69} for virtually any
functional of any univariate probability distribution.  Other
fundamental early contributions in this topic are
\cite{chen1985poincare}, \cite{CP86,CaPa85}, or
\cite{karlin1993general} wherein various extensions are proposed
(e.g.\ Karlin deals with the entire class of log-concave
distributions).  A major breakthrough is due to
\cite{papathanasiou1988variance} who obtains infinite expansions for
continuous targets, with coefficients very close in spirit to those
that we shall propose in Section \ref{sec:variance-expansion}.
Papathanasiou's method of proof in \cite{papathanasiou1988variance} --
which rests in an iterative rewriting of the exact remainder in the
Cauchy Schwarz identity -- is also a direct inspiration for ours.
Such results open the way for a succesful line of research in
connection with Pearson's and Ord's system of distributions including
works such as \cite{CP89}, \cite{korwar1991characterizations},
\cite{johnson1993note}, \cite{papathanasiou1995characterization}, and
\cite{CP95}. Similar results, by different means, are Houdr\'e and
Kagan's famous arbitrary order bound \eqref{eq:71} from \cite{HK95},
extended to wide families of targets e.g.\ in
\cite{houdre1998interpolation,HP95}.
The remarkable articles \cite{DZ91} and \cite{Le95} both propose
similar minded considerations in very general settings. More recently,
the contributions \cite{APP07}, \cite{APP11} \cite{ap14} and
\cite{AfBalPa2014} and \cite{LS12a,LS16} begin to fully explore
connections with Stein's method.  In the Gaussian framework, an
enlightening first order matrix variance bound is proposed in
\cite{olkin2005matrix} for the Gaussian distribution, and in a more
general setting in \cite{afendras2011matrix} (and also to arbitrary
order). Finally, we mention \cite{saumard2018efron,
  saumard2017isoperimetric} and \cite{saumard2018weighted}'s
revisiting of this classical literature enticed us to begin the work
that ultimately led to the present paper. We end this literature
review by pointing to \cite{courtade2017existence,fathi2018stein,
  fathi2018higher} wherein fundamental contributions to the theory of
Stein kernels are provided in a multivariate setting. 

\subsection{Outline of the paper and main results}
\label{sec:outline-paper-main}

In Section \ref{sec:stein-diff} we recall the theory of canonical and
standardized Stein operators introduced in \cite{LRS16}, and adapt it
to our present framework. We introduce a (new) notion of inverse Stein
operator (Definition \ref{def:can_inverse}) along with a first
representation formula \eqref{eq:17}; we also introduce one of the
basic tools of the paper, namely a generalized indicator function
\eqref{eq:kerneldeff} which will serve throughout the paper. 

In
Section \ref{sec:covar-ident} we present the main covariance identities
(Lemmas \ref{lem:stIBP1}, \ref{lma:repformIII}) along with a second
probabilistic representation of the inverse operator in
\eqref{eq:24}. The first variance Klaassen-type bounds are provided in
Theorem \ref{thm:klaassenbounds}. They take on the same exact form as
\eqref{eq:69}. We stress that the method of proof is new, and
elementary.  Many standard examples are detailed. 

Section 4 begins with a third probabilistic representation of the
inverse operator in \eqref{eq:reprsntform}, which despite its
simplicity seems to be new.  A re-interpretation of the classical
Lagrange identity is provided in Lemma \ref{lma:prob}, in our
formalism; the basic building blocks of the theory are given in
\eqref{eq:29} and \eqref{eq:51}; these definitions are exactly the
correct tool for obtaining general arbitrary order variance expansions
detailed in Theorem~\ref{thm:var-}. We identify in \eqref{eq:2}
fundamental ``iterated Stein coefficients'' which we denote
$\Gamma_k^{\ell}(x)$; these generalize the classical Stein kernel and
we give some abstract examples in Lemma \ref{lma:gamma}, with concrete
examples also provided for several classical
targets. Section~\ref{sec:mult-extens-olkin} details our extension of
Olkin and Shepp's first-order matrix variance-inequality ``\`a la
Klaassen''. The final section, Section~\ref{sec:new-repr-stein},
applies our framework to obtain bounds on solutions of Stein
equations.

\section{Stein differentiation}
\label{sec:stein-diff}

Let $\mathcal{X} \subset \R$ and equip it with some $\sigma$-algebra
$\mathcal{A}$ and $\sigma$-finite measure $\mu$. Let $X$ be a random
variable on $\mathcal{X}$, with probability measure $P^X$ which is absolutely continuous with respect to $\mu$; we denote
$p$ the corresponding density, and
$\mathcal{S}(p) = \left\{ x \in \mathcal{X} : p(x)>0\right\}$.
Although we could in principle keep the discussion to come very
general, in order to make the paper more concrete and readable in the
sequel we shall restrict our attention to distributions satisfying the
following assumption.   

\

\noindent \textbf{Assumption A.} The measure $\mu$ is either the
counting measure on $\mathcal{X} = \Z$ or the Lebesgue measure on
$\mathcal{X} = \R$. If $\mu$ is the counting measure then there exist
$a, b \in \Z \cup \left\{-\infty, \infty \right\}$ such that
$\mathcal{S}(p) = [a, b]\cap \N$.  If $\mu$ is the Lebesgue measure
then $\overline{\mathcal{S}(p)} = [a, b]$ for
$a, b \in \R\cup \left\{ -\infty, \infty \right\}$.

\

{{The necessary feature of {Assumption A}  is that $\mu$ is translation-invariant; the assumption of compact support is made for convenience. }}

Let $\Delta^{\ell}f(x) = (f(x+\ell) - f(x))/\ell$ for all
$\ell \in \R$, with the convention that $\Delta^0f(x) = f'(x)$, with
$f'(x)$ the weak derivative defined Lebesgue almost everywhere.  We
denote $\mathrm{dom}(p, \Delta^{\ell})$ the collection of functions
$f : \R \to \R$ such that $\Delta^{\ell}f(x)$ exists and is finite at
all $x \in \mathcal{S}(p)$.  The next definitions come from
\cite{LRS16}.

\begin{defn}[Canonical Stein operators] \label{def:stei_op} Let
  $f \in \mathrm{dom}(p, \Delta^{\ell})$ and consider
  $\mathcal{T}_p^{\ell}f$ defined as
$  \mathcal{T}_p^{\ell} f(x) =  \frac{\Delta^\ell (f(x) p(x))}{p(x)}$
for all $x \in \mathcal{S}(p)$ and $ \mathcal{T}_p^{\ell} f(x) = 0$ for
$x \notin \mathcal{S}(p)$. Operator $\mathcal{T}_p^{\ell}$ is the
canonical ($\ell$-)Stein operator of $p$.  The cases $\ell=1$ and
$\ell = -1$ provide the {forward} and {backward} Stein operators,
denoted $\mathcal{T}_p^+$ and $\mathcal{T}_p^-$, respectively; the
case $\ell=0$ provides the differential Stein operator denoted
$\mathcal{T}_p$.
\end{defn}

\begin{defn}[Canonical Stein class]\label{def:can_class} We define
  $\mathcal{F}^{(0)}(p)$ as the class of functions in $L^1(p)$ which
  moreover have mean 0 with respect to $p$.  The canonical
  $\ell$-Stein class of $p$ is the collection
  $\mathcal{F}^{(1)}_{\ell}(p)$ of $f : \R \to \R$ such that
  $x \mapsto f(x) p(x) \in \mathrm{dom}(p, \Delta^{\ell})$ and
  $\mathcal{T}_p^{\ell} f \in \mathcal{F}^{(0)}(p)$ (i.e.\
  $\mathbb{E} \left[ \mathcal{T}_p^{\ell} f (X) \right]=0$).
\end{defn}  
\begin{rmk}
  For details on Stein class and operators, we refer to (i)
  \cite{LRS16} for the construction in an abstract setting, (ii)
  \cite{ley2015distances} for the construction in the continuous setting (i.e.\
  $\ell = 0$) and (iii) \cite{ES18} for the construction in the
  discrete setting (i.e.\ $\ell \in \left\{ -1, 1 \right\}$).

\end{rmk}

{In what follows we restrict attention to the following cases: 
$\ell =0$ and $\mu$ is the Lebesgue measure; or $\ell \in \{ -1, +1\}
$ and $\mu$ is counting measure. We note that for positive integer
$\ell$, we have  
$ \Delta^\ell f(x) = \Delta^1 \left( \frac{1}{\ell}
  \sum_{\ell=0}^{k-1} f(x+k) \right)$ and with this representation
results for other integer values of $\ell$ could be obtained. For
simplicity we do not consider this general case. }

By definition, the canonical operators $\mathcal{T}_p^{\ell}$ embed
$ \mathcal{F}_{\ell}^{(1)}(p)$ into $\mathcal{F}^{(0)}(p)$ (and in
particular $\mathbb{E} [ \mathcal{T}_p^{\ell}f(X)]= 0$ for all
$f \in \mathcal{F}_{\ell}^{(1)}(p)$). Inverting this operation is an
important part of the construction to come; the following result is
well-known and easy to obtain.

\begin{lma}[Representation formula I]\label{lem:can_inverse}
  Let $\ell \in \left\{ -1, 0, 1 \right\}$.  Let $p$ with support
  $\mathcal{S}(p)$  satisfy Assumption A and define 
 \begin{equation}
   \label{eq:kerneldeff}
   \chi^{\ell}(x, y) = \mathbb{I}_{x \le y - \ell (\ell
     +1)/2}. 
 \end{equation}
 Then for any $ \eta \in \mathcal{F}^{(0)}(p)$, the function defined
 on $\mathcal{S}(p)$ as
  \begin{equation}
    \label{eq:17}
 \mathcal{L}^{\ell}_p\eta(x)    = \frac{1}{p(x)}\mathbb{E} \left[
   \chi^{\ell}( X, x) \eta(X)\right] = \frac{1}{p(x)}\int
 \chi^{\ell}(y, x)
 \eta(y) 
  p(y) d\mu(y)
  \end{equation}
  satisfies (i)
  $\mathcal{L}^{\ell}_p\eta \in \mathcal{F}_{\ell}^{(1)}(p)$ and (ii)
  $ \mathcal{T}_{p}^{\ell}\mathcal{L}^{\ell}_p\eta(x) = \eta(x)$ for
  all $x \in \mathcal{S}(p)$. If furthermore
  $\eta \in \mathcal{F}_{\ell}^{(1)}(p)$ then also
  $ \mathcal{L}_{p}^{\ell}\mathcal{T}^{\ell}_p\eta(x) = \eta(x)$ for
  all $x \in \mathcal{S}(p)$.
\end{lma}

{\begin{rmk} For $f \in {\rm{dom}}(p, \Delta^\ell)$ such that $ \mathbb{E} f(X) =0$  it holds that 
 $$
 \mathcal{L}^{\ell}_p f(x)    = \frac{1}{p(x)}\mathbb{E} \left[
   \left( 1 - \chi^{\ell}( X, x) \right) f(X)\right] = \frac{1}{p(x)}\int
\left( 1 -  \chi^{\ell}(y, x) \right) 
f(y) 
  p(y) d\mu(y). 
$$
Hence instead of  Definition \eqref{eq:kerneldeff} in the developments below a more general version 
$$
   \zeta^{\ell}(x, y) = \alpha \,  \mathbb{I}_{x \le y - \ell (\ell
     +1)/2} - \beta \,  \mathbb{I}_{x > y - \ell (\ell
     +1)/2} 
$$
could be used, where $\alpha + \beta =1$. 
For simplicity we use Definition \eqref{eq:kerneldeff}.
\end{rmk}

The expressions in Lemma \ref{lem:can_inverse} particularize, in the
three cases that interest us, to
$\chi^-(x, y) = \mathbb{I}_{x \le y}$ $\big(\ell = -1\big)$,
$ \chi^+(x, y) = \mathbb{I}_{x < y} $ $\big(\ell = 1\big)$ and
$\chi^0(x, y) = \mathbb{I}_{x \le y}$ $\big(\ell = 0\big)$.  
If
$\mu$ is the Lebesgue measure (and $\ell = 0$) then
\begin{equation}
  \label{eq:1}
  \mathcal{L}_p^0\eta(x) = \frac{1}{p(x)} \int_a^x \eta(u) p(u) du
\end{equation}
 whereas if
$\mu$ is the counting measure then 
\begin{align}
  \label{eq:4}
\mathcal{L}_p^+\eta(x) = \frac{1}{p(x)}
\sum_{j=a}^{x-1} \eta(j) p(j)  \mbox{ }(\ell=1)  \mbox{ and }
\mathcal{L}_p^-\eta(x) =
\frac{1}{p(x)} \sum_{j=a}^{x} \eta(j) p(j) \mbox{ }(\ell=-1).
\end{align}
In all cases the functions are extended on $\mathcal{X}$ through the
convention that $ \mathcal{L}_p^{\ell}\eta(x) = 0$ for all
$x \notin \mathcal{S}(p)$.

\begin{defn}[Canonical pseudo inverse Stein
  operator]\label{def:can_inverse}
The canonical pseudo-inverse Stein
   operator  is 
   \begin{equation}
     \label{eq:53}
     \mathcal{L}_p^{\ell} : L^1(p) \to
     \mathcal{F}^{(1)}_{\ell}(p) : h \mapsto \mathcal{L}^{\ell}_p(h -
     \mathbb{E}[h(X)]) 
   \end{equation}
with $\mathcal{L}^{\ell}_ph(x)$ defined in Lemma
\ref{lem:can_inverse}. 
\end{defn}

Since the operators $\Delta^{\ell}$ satisfy the product rule
$\Delta^{\ell}\big(f(x)g(x-\ell)\big) = \Delta^{\ell} f(x) g(x) + f(x)
\Delta^{-\ell}g(x)$, we immediately obtain that
$\mathcal{T}_p^{\ell} \big(f(x) g(x-\ell)\big) = \mathcal{T}_p^{\ell}f(x)
g(x) + f(x) \Delta^{-\ell}g(x)$ for all appropriate $f,
g$. This leads to
\begin{defn}[Standardizations of the canonical
  operator]\label{def:stand-canon-oper}
  Fix some $\eta \in \mathcal{F}^{(0)}(p)$. The $\eta$-standardized
  Stein operator is 
\begin{align}
  \label{eq:33}
  \mathcal{A}_pg(x) & = \mathcal{T}^{\ell}_p\big(\mathcal{L}^{\ell}_p\eta(\cdot)
                      g(\cdot-\ell)\big)(x) =  \eta(x)  g(x) +
                      \mathcal{L}^{\ell}_p\eta(x) 
                      \big(\Delta^{-\ell}g(x)\big) 
\end{align}
acting on the collection $\mathcal{F}(\mathcal{A}_p)$ of test
functions $g$ such that
$ \eta(\cdot)g(\cdot - \ell) \in \mathcal{F}_{\ell}^{(1)}(p)$.
\end{defn}

\begin{exm}[Binomial distribution]\label{ex:binom1}
Stein's method for the binomial distribution was first developed in \cite{ehm1991binomial}.   Let $p$ be the binomial density with parameters $(n, p)$. Then
  $\mathcal{S}(p) = [0,n]$, $\mathcal{F}_+^{(1)}(p)$ consists of bounded
  functions such that $f(0) = 0$ and $\mathcal{F}_-^{(1)}(p)$ of functions
  such that $f(n) = 0$.  Picking $\ell = 1$ and $\eta(x) = x-np$ then
  $\mathcal{L}_{n, p}^+ \eta(x) = -(1-p)x$ so that
    \begin{align}
      \label{eq:34}
      \mathcal{A}^1_{\mathrm{bin}(n, p)}g(x) =  (x-np)g(x) - (1-p)x\Delta^-g(x)
    \end{align}
    with corresponding class
    $ \mathcal{F}(\mathcal{A}_{\mathrm{bin}(n, p)})$ which contains
    all bounded functions $g : \Z \to \R$.  Picking $\ell = -1$ and
    $\eta(x) = x-np$ then $\mathcal{L}_{n, p}^- \eta(x) = - p(n-x)$ so
    that
    \begin{align}
\label{eq:35}
      \mathcal{A}^2_{\mathrm{bin}(n, p)}g(x) = 
      (x-np)g(x) - p(n-x)\Delta^+g(x)
    \end{align}
    acting on the same class as \eqref{eq:34}. 
\end{exm}

\begin{exm}[Poisson distribution]\label{ex:poi1}
Stein's method for the Poisson distribution originates in\cite{C75}. 
  Let $p$ be the Poisson density with parameter $\lambda$. Then
  $\mathcal{S}(p) = \N$, $\mathcal{F}_+(p)$ consists of bounded functions
  such that $f(0) = 0$ and $\mathcal{F}_-(p)$ of functions such that
  $\lim_{n\to \infty}f(n)\lambda^n/n!= 0$. 
Picking $\ell = 1$ and $ \eta(x) = x-\lambda$ then
      $\mathcal{L}_{ \lambda}^+ \eta(x) = -x$  so that 
    \begin{align}
\label{eq:41Poisson}
      \mathcal{A}^1_{\mathrm{Poi}(\lambda)}g(x) =  (x-\lambda) g(x) -x
      \Delta^-g(x) 
    \end{align}
    acting on $ \mathcal{F}(\mathcal{A}^1_{\mathrm{Poi}(\lambda)})$
    which contains all bounded functions $g : \Z \to \R$. 
 Picking $\ell = -1$ and  $\eta(x) = x-\lambda$ then
    $\mathcal{L}_{\lambda}^- \eta(x) = -\lambda$ so that 
    \begin{align}
\label{eq:40}
      \mathcal{A}^2_{\mathrm{Poi}( \lambda)}g(x) = (x-\lambda)g(x) -
      \lambda \Delta^+g(x)
    \end{align}
    acting on the same class as \eqref{eq:41Poisson}.

\end{exm}

\begin{exm}[Beta distribution]\label{ex:beta1}
  Let $p$ be the Beta density with parameters $(\alpha, \beta)$. Then
  $\mathcal{S}(p) = [0,1]$ and $\mathcal{F}(p)$ consists of 
  functions such that $f(0) = f(1) = 0$.  If
  $\eta(x) = x - \alpha/(\alpha+\beta)$ then
  $\mathcal{L}_{\alpha, \beta}\eta(x) = -\frac{x(1-x)}{\alpha + \beta}$ leading to
  the operator
    \begin{align}
\label{eq:41beta}
      \mathcal{A}_{\mathrm{Beta}(\alpha, \beta)}g(x) =
      \Big(x-\frac{\alpha}{\alpha+\beta}\Big) g(x) -\frac{
      x(1-x)}{\alpha+\beta}  g'(x)
    \end{align}
    with corresponding class
    $ \mathcal{F}(\mathcal{A}_{\mathrm{Beta}(\alpha, \beta)})$ which
    contains all {{differentiable, bounded}} functions
    $g : \R \to \R$; {{see \cite{Do14} and \cite{goldstein2013stein}
        for more details and related Stein operators.}}
\end{exm}

\begin{exm}[Gamma distribution]\label{ex:gamma1}
  Let $p$ be the gamma density with parameters $(\alpha, \beta)$. Then
  $\mathcal{S}(p) = (0, \infty)$ ($\alpha<1$) or
  $\mathcal{S}(p) = [0, \infty)$ ($\alpha \ge 1$) and $\mathcal{F}(p)$
  consists of functions such that $f(0) = 0$ all integrals exist.  If
  $\eta(x) = x - \alpha\beta$ then $\mathcal{L}_p\eta(x) = - \beta x$
  leading to the operator
    \begin{align}
\label{eq:41gamma}
      \mathcal{A}_{\mathrm{gamma}(\alpha, \beta)}g(x) =
      \big(x-\alpha\beta\big) g(x) -\beta x  g'(x)
    \end{align}
    with corresponding class
    $ \mathcal{F}(\mathcal{A}_{\mathrm{Gamma}(\alpha, \beta)})$ which
    contains all {{differentiable}} functions $g : \R \to \R$ such
    that {{the function $x g(x)$ is bounded; see for example
        \cite{gaunt2015chi} and \cite{Luk1994} for more details and
        related Stein operators}}. 
\end{exm}

\begin{exm}[A general example]\label{ex:genera}
  Let $p$ satisfy Assumption A  and suppose that it has finite mean
  $\mu$.   If
  $\eta(x) = x - \mu$ then
  $\mathcal{L}_p\eta(x) = -\tau_p(x)$ is the so-called \emph{Stein
    kernel} of $p$,  leading to
  the operator
    \begin{align}
\label{eq:41}
      \mathcal{A}_pg(x) =
      \big(x-\mu\big) g(x) - \tau_p(x) g'(x)
    \end{align}
    with corresponding class $ \mathcal{F}(\mathcal{A}_p)$ which
    contains all functions $g : \R \to \R$ such that
    $\tau_p g \in \mathcal{F}(p)$.  {{ Hence in the general case
        \eqref{eq:33} can be viewed as generalisation of the Stein
        kernel.  Implications of this observations are explored in
        detail in \cite{ES18}. }}
\end{exm}

\section{Stein covariance identities and Klaassen-type bounds}
\label{sec:covar-ident}

The construction from the previous  sections is tailored to ensure
that we easily deduce the following family of ``Stein covariance
identities'' (e.g.\ use $\mathbb{E}[ \mathcal{A}_p^1g(X)] = 0$ for
  $\mathcal{A}_p^1$ as defined in
  \eqref{eq:33}). {{Lemma~\ref{lem:stIBP1} follows directly from
      the Stein product rule  in \cite[Theorem 3.24]{LRS16}. 
\begin{lma}[Stein IBP formulas]
\label{lem:stIBP1}
Let $X \sim p$. Then, for all $c \in \mathcal{F}_{\ell}^{(1)}(p)$ on
the one hand, and all $\eta \in \mathcal{F}^{(0)}(p)$ on the other
hand, we have
\begin{align}
  \label{eq:20}
&   \mathbb{E} \left[ c(X) \Delta^{-\ell}g(X) \right] =- \mathbb{E}
  \left[ {{ \{}} \mathcal{T}_p^{\ell}c(X) {{ \}}}g(X) \right] \\
  \label{eq:21}
  & \mathbb{E} \left[ - {{ \{}}  \mathcal{L}_p^{\ell}\eta(X) {{ \}}}\Delta^{-\ell}g(X)
    \right] = \mathbb{E} 
  \left[\eta(X) g(X) \right] 
\end{align}
for all $g$ such that
$c(\cdot)g(\cdot - \ell) \in \mathcal{F}_{\ell}^{(1)}(p)$ (identity
\eqref{eq:20}) and all $g$ such that
$\mathcal{L}_p^{\ell}\eta(\cdot)g(\cdot - \ell) \in
\mathcal{F}_{\ell}^{(1)}(p)$ (identity \eqref{eq:21}). 
\end{lma}
These identities provide powerful handles on the target distribution
$X\sim p$. Combining them for instance with a second representation
for Stein operator $\mathcal{L}_p^{\ell}$ given in \eqref{eq:24} leads
to the  following result.

\begin{lma}[Representation formula II]\label{lma:repformIII}
Define
\begin{equation}
  \label{eq:15}
  K_p^{\ell}(x, x')  = 
\mathbb{E} [\chi^{\ell}(X, x)
    \chi^{\ell}(X, x') \big]  -
    \mathbb{E} [\chi^{\ell}(X, x)\big] \mathbb{E}\big[\chi^{\ell}(X, x')
    \big]. 
\end{equation} 
Then
\begin{align}
     \mathrm{Cov}[h(X), g(X)] & = \mathbb{E}
     \bigg[-\mathcal{L}_p^{\ell}h(X) \Delta^{-\ell}g(X)
                                \bigg] \label{eq:25} \\
& = \mathbb{E} \left[ \Delta^{-\ell} h(X)
     \frac{ K_p^{\ell}(X, X')}{p(X)p(X')}\Delta^{-\ell} g(X')\right]  \label{eq:14}\\
  \label{eq:18}
  &\bigg( = \int_{\mathcal{S}(p)}\int_{\mathcal{S}(p)}  \Delta^{-\ell} h(x)
    K_p^{\ell}(x, 
    x')  \Delta^{-\ell} g(x') d\mu(x)  d\mu(x')  \bigg)
   \end{align}
  for all $h \in L^1(p)$ and $g$ such that {the integrals
     exist}. Moreover, $ K_p^{\ell}(x, x') $ is positive and bounded.
     If $\mathcal{F}^{(0)}(p)$ is dense in  $ L^1(p)$, then 
  \begin{equation}
\label{eq:24}
    -\mathcal{L}_p^{\ell}h(x')  = \mathbb{E}\left[ \frac{
                                  K_p^{\ell}(X, x')}{p(X)p(x')}
                                  \Delta^{-\ell} h(X) \right]
  \end{equation}
for all $h \in L^1(p)$.    
\end{lma} 
\begin{proof}
  Let $\bar{h}(x) = h(x) - \mathbb{E}[h(X)]$. {{Note that $\Delta^{\ell} {\bar{h}} = \Delta^\ell h.$}} 
 Equality
  \eqref{eq:25} follows from \eqref{eq:21} and the fact that
  $ \mathrm{Cov}[h(X), g(X)) = \mathbb{E}[(h(X) - \mathbb{E}[h(X)])
  g(X)]$. To obtain \eqref{eq:14} we start from \eqref{eq:25} and use
  representation \eqref{eq:17}:
\begin{align*}
  \mathrm{Cov}[h(X), g(X)]  & = \mathbb{E}
                              \bigg[-\mathcal{L}_p^{\ell}\bar{h}(X) \Delta^{-\ell}g(X) \bigg]  
                            = -  \mathbb{E}
                              \bigg[\bar{h}(X') \frac{\chi^{\ell}(X', X) }{p(X)} \Delta^{-\ell}g(X) \bigg]
\end{align*}
with $X'$ an independent copy of $X$. 
{{Taking conditional expectations gives
\begin{align*}
  \mathrm{Cov}[h(X), g(X)]  & -  \mathbb{E}
                              \bigg[\mathbb{E} \bigg[ \bar{h}(X') \chi^{\ell}(X', X)  | X \bigg]  \frac{1}{p(X)} \Delta^{-\ell}g(X) \bigg].
\end{align*}
}} 
 Next, {{let $\rho_{x}(x')  = \chi^\ell(x',x)$ and view this as a function of $x'$, to obtain 
 from \eqref{eq:21}, 
 $$  \mathbb{E} \bigg[\rho_{x}(X')  \bar{h}(X') \bigg]  =  \mathbb{E} \bigg[{\bar{\rho_{x}}} (X')  \bar{h}(X') \bigg] 
 =  \mathbb{E} \bigg[ - (\mathcal{L}_p^\ell {\bar{\rho_{x}}}(X')) \Delta^{-\ell} {h}(X') \bigg] 
.
 $$
 }}Denoting
$ \mathcal{L}_{p, x'}^{\ell}$ as operator $\mathcal{L}_p^{\ell}$
acting on the {{first}} component, {{indexed by}} $x'$, 
we 
get 
\begin{align}
  \mathrm{Cov}[h(X), g(X)]  
                            &    =  \mathbb{E}
                              \bigg[ \Delta^{-\ell}h(X')\frac{
                              \mathcal{L}_{p, X'}^{\ell}(
                              \chi^{\ell}(X', X))}{p(X)}  
                              \Delta^{-\ell}g(X)    \bigg].\label{eq:23}
\end{align}
Using again
\eqref{eq:17} we have
\begin{align}
  \mathcal{L}_{p, x'}^{\ell}( \chi^{\ell}(x', x)) & = 
                                                    \frac{1}{ p(x') }   \mathbb{E} [\chi^{\ell}(X'', x') \big(\chi^{\ell}(X'', x) -
                                                    \mathbb{E}[\chi^{\ell}(X'',
                                                    x)]\big)  ]
                                                    \nonumber \\
                                                  &   \label{eq:22} 
                                                    = \frac{1}{p(x')}\big(\mathbb{E} [\chi^{\ell}(X'', x') \chi^{\ell}(X'', x) \big] -
                                                    \mathbb{E} [\chi^{\ell}(X'', x')\big] \mathbb{E}\big[\chi^{\ell}(X'',
                                                    x) \big] \big)  
\end{align}
with $X''$ another independent copy of $X$.   Plugging this last
identity into the covariance representation \eqref{eq:23} we obtain the
second equality in \eqref{eq:14}. To see that $K_p^{\ell}(x, x') \ge
0$ for all $x, x'$ it suffices to notice that $x'\mapsto
\chi^{\ell}(x', x)$ is decreasing and we can apply forthcoming
Proposition \ref{prop:applirep}. {{
Alternatively,  is suffices to note that for all $x,y$,
$$\mathbb{P} (X \le x \wedge y) - \mathbb{P} (X \le x )\mathbb{P} (X \le  y)= 
\mathbb{P} (X \le x \wedge y) ( 1 - \mathbb{P} (X \le  x \vee y)) \ge 0.
$$
}}{{Equality~\ref{eq:24} follows  from  \eqref{eq:14} and \eqref{eq:17} and the assumption that $\mathcal{F}^{(0)}(p)$ is dense in  $ L^1(p)$.}}
\end{proof}

\begin{exm}[Binomial distribution]\label{ex:binom2}
From previous developments we get 
\begin{equation}
  \label{eq:26}
\mathrm{Cov}[X, g(X)] = \mathbb{E} \left[ (1-p)X
    \Delta^{-} g(X) \right] = p  \mathbb{E} \left[ (n-X)
    \Delta^+g(X)\right] 
\end{equation}
which is valid for all  functions $g : \Z \to \R$ that are bounded on
$[0, n]$. Combining the two identities   we also arrive at 
\begin{equation}\label{eq:37}
\mathrm{Cov}[X, g(X)]= \mathrm{Var}[X]
  \mathbb{E}[\nabla_{\mathrm{bin}(n, p)}g(X)]
\end{equation}
with $\nabla_{\mathrm{bin}(n, p)}$ the ``natural'' binomial gradient  $\nabla_{\mathrm{bin}(n, p)} g(x) =
\frac{x}{n} \Delta^-g(x) + 
\frac{n-x}{n} 
\Delta^+g(x)$ from \cite{hillion2011natural}. 
\end{exm}

\begin{exm}[Poisson distribution]\label{ex:poiss2}
From previous developments we get 
\begin{equation}
  \label{eq:3}
\mathrm{Cov}[X, g(X)] = \mathbb{E} \left[ X
    \Delta^-g(X) \right] = \mathbb{E} \left[
      \lambda \Delta^+g(X) \right]
\end{equation}
for all functions on $\Z$ such that $p(x) g(x) \to 0$ as
$x\to \infty$.  Similarly as in example \ref{ex:binom2} we combine the
two identities to obtain
\begin{equation}
  \label{eq:36}
\mathrm{Cov}[X, g(X)] =  \mathrm{Var}[X] \mathbb{E} \left[
    \nabla_{\mathrm{Poi}(\lambda)} g(X) \right], 
\end{equation}
this time with $\nabla_{\mathrm{Poi}(\lambda)} g(x) = \frac{1}{2}\left(\frac{x}{\lambda} \Delta^-g(x) +
\Delta^+g(x)\right)$. 
\end{exm}

\begin{exm}[Beta distribution]\label{ex:beta2}
From previous developments we get 
\begin{equation}
  \label{eq:10}
\mathrm{Cov}[X, g(X)] = \frac{1}{\alpha+\beta}  \mathbb{E} \left[ X (1-X)
    g'(X) \right]
\end{equation}
\end{exm}

\begin{rmk}
Identity \eqref{eq:22} and the kernel $K_p(x, x')$ in \eqref{eq:15}
will turn out very useful in future developments. It is informative to
particularize this kernel in three cases of interest to us: letting
$P$ denote the cdf of $p$ and $\min(x, x') = x \wedge x'$, we have
\begin{align*}
K_p^{0}(x, x') & = K_p^-(x, x') =  P(x \wedge x') -
                          P(x) P(x') \\
K_p^+(x, x') & = P(x\wedge x'-1) - P(x-1) P(x'-1). 
\end{align*}
Several illustrations are provided in Figure \ref{fig:figKPP}.

\end{rmk}

 \begin{figure}
\begin{subfigure}{.33\textwidth}
  \centering
  \includegraphics[width=.8\linewidth]{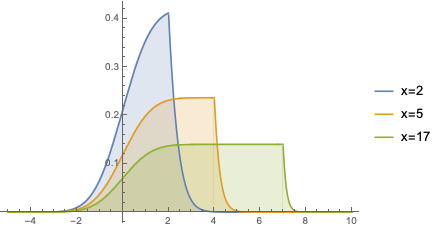}
  \caption{}
  \label{fig:sfig11}
\end{subfigure}%
\begin{subfigure}{.33\textwidth}
  \centering
  \includegraphics[width=.8\linewidth]{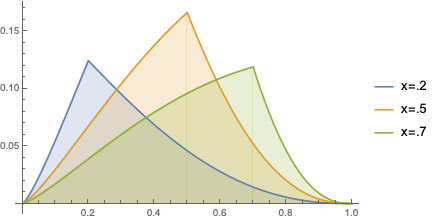}
  \caption{}
  \label{fig:sfig21}
\end{subfigure}
\begin{subfigure}{.33\textwidth}
  \centering
  \includegraphics[width=.8\linewidth]{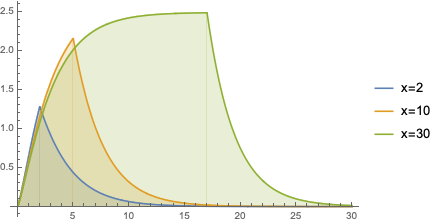}
  \caption{}
  \label{fig:sfig31}
\end{subfigure}

\begin{subfigure}{.33\textwidth}
  \centering
  \includegraphics[width=.8\linewidth]{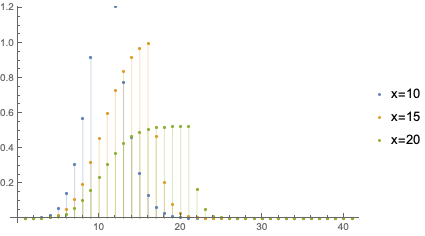}
  \caption{}
  \label{fig:sfigbin1}
\end{subfigure}%
\begin{subfigure}{.33\textwidth}
  \centering
  \includegraphics[width=.8\linewidth]{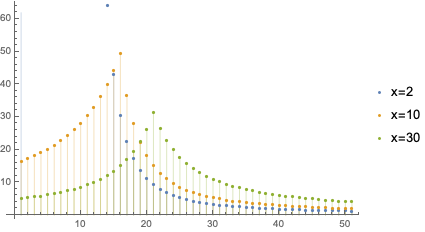}
  \caption{}
  \label{fig:sfigpoi1}
\end{subfigure}
\begin{subfigure}{.3\textwidth}
  \centering
  \includegraphics[width=.8\linewidth]{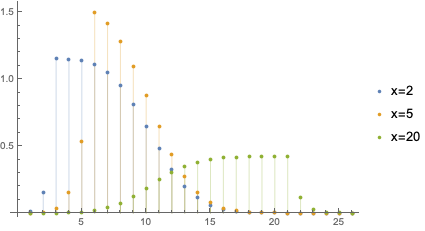}
  \caption{}
  \label{fig:sfighyp1}
\end{subfigure}
\caption{{\it The functions
    $x' \mapsto \mathcal{L}_{p, x'}^{\ell}( \chi^{\ell}(x', x)) =
    K_p^{\ell}(x, x')/p(x)$ for different (fixed) values of $x$ and
    $p$ the standard normal distribution (Figure~\ref{fig:sfig11});
    Beta distribution with parameters 1.3 and 2.4
    (Figure~\ref{fig:sfig21}); Gamma distribution with parameters 1.3
    and 2.4 (Figure~\ref{fig:sfig31}); binomial distribution with
    parameters $(20, .2)$ (Figure~\ref{fig:sfigbin1}); Poisson
    distribution with parameter 15 (Figure~\ref{fig:sfigpoi1});
    hypergeometric distribution with parameters 100, 50 and 500
    (Figure~\ref{fig:sfighyp1})}.
\label{fig:figKPP}}
\end{figure}

Theorem \ref{thm:klaassenbounds} gives upper and lower bounds of
functionals of random variables with coefficients (weights) expressed
in terms of Stein operators. This result provides so-called weighted
Poincar\'e inequalities such as those described in
\cite{saumard2018weighted}; it also contains the main result from
\cite{K85}. While the lower bound -- some type of Cramer-Rao bound --
has long been known to have a very simple proof via Stein operators, a
proof for the upper bounds has always been either elusive or rather
complicated outside of the gaussian case (see e.g.\ \cite{saumard2018weighted}). The
representation formulae obtained for Stein integral operators will turn out to simplify
the work considerably.

\begin{thm}[Klaassen bounds] \label{thm:klaassenbounds} For any
  $c \in \mathcal{F}_{\ell}^{(1)}(p)$ we have for all $f$ such that
{$\bar{f}^2\in L^1(p)$}:
 \begin{equation}
\label{eq:13}
 \mathrm{Var}[f(X)]  \ge \frac{\mathbb{E}\left[ c(X) (\Delta^{-\ell}f(X))\right]^{2}}{\mathbb{E} \left[
\big(\mathcal{T}_p^{\ell}c(X)\big)^2 \right]}
  \end{equation}
  with equality if and only if there exist $\alpha, \beta$ real
  numbers such that 
  $f(x) = \alpha \mathcal{T}_p^{\ell}c(x) + \beta$ for all
  $x \in \mathcal{S}(p)$.   Moreover,  if  $h \in L^1(p)$ is
  decreasing then for all $f$ such that
  $\mathcal{L}_p^{\ell}h(\cdot) f(\cdot-\ell) \in
  \mathcal{F}^{(1)}_\ell(p)$ we have
  \begin{equation}
    \label{eq:19}
  \mathrm{Var}[f(X)] \le \mathbb{E}\left[ (\Delta^{-\ell}f(X))^2
      \frac{-\mathcal{L}_p^{\ell}h(X)}{\Delta^{-\ell}h(X)} \right]
  \end{equation}
  with equality if and only if there exist $\alpha, \beta$ real
  numbers such that 
  $f(x) = \alpha h(x) + \beta$ for all
  $x \in \mathcal{S}(p)$. 
\end{thm}
\begin{rmk}
  The requirement that $h$ be decreasing is not necessary; all that is
  needed is in fact that
  $ {-\mathcal{L}_p^{\ell}h(X)}/{\Delta^{-\ell}h(X)}>0$ almost surely.
\end{rmk}

\begin{proof}[Proof of Theorem \ref{thm:klaassenbounds}]
  The lower bound \eqref{eq:13} is an almost immediate consequence of
  \eqref{eq:20}. Indeed, using the fact that
  $\mathcal{T}_p^{\ell} c(X) \in \mathcal{F}^{(0)}(p)$ for all
  $c \in \mathcal{F}^{\ell}(p)$ we have
\begin{align*}
 \mathbb{E} \left[  c(X) \big(\Delta^{-\ell} f(X) \big) \right]^2 & =
 \mathbb{E}
 \left[
 \big(\mathcal{T}_p^\ell
 c(X)
  \big)
 f(X)
\right]^2  
 = \mathbb{E} 
  \left[
  \big(\mathcal{T}_p^\ell
   c(X)
 \big)
\bar{f}(X)
 \right]^2 \\
& \le \mathbb{E} \left[
\big(\mathcal{T}_p^{\ell}c(X)\big)^2 \right] \mathrm{Var}[f(X)]
\end{align*}
by the Cauchy-Schwarz inequality. 
To see \eqref{eq:19} we simply apply \eqref{eq:14} and the
Cauchy-Schwarz inequality to obtain
\begin{align*}
 &  \mathrm{Var}[f(X)]  = \mathbb{E} \left[ \bar{f}(X) f(X) \right]  =
       \mathbb{E} \left[\Delta^{-\ell} f(X)   \frac{K_p^{\ell}(X, X')}{p(X)p(X')} \Delta^{-\ell} f(X')
       \right] \\
  & =  \mathbb{E} \left[ \left\{ \frac{\Delta^{-\ell}
    f(X)}{\sqrt{\Delta^{-\ell} h(X)}}  \sqrt{\frac{K_p^{\ell}(X,
    X')}{p(X)p(X')} \Delta^{-\ell}h(X')}  \right\}   \left\{ \frac{\Delta^{-\ell}
    f(X')}{\sqrt{\Delta^{-\ell} h(X')}}  \sqrt{\frac{K_p^{\ell}(X,
    X')}{p(X)p(X')} \Delta^{-\ell}h(X)} \right\}\right]\\
& \le \mathbb{E} \left[ \frac{(\Delta^{-\ell}
    f(X))^2}{\Delta^{-\ell} h(X)} \frac{K_p^{\ell}(X,
    X')}{p(X)p(X')} \Delta^{-\ell}h(X') \right];
\end{align*}
using \eqref{eq:24} finishes the proof. 
\end{proof}

\begin{exm}[Binomial bounds]\label{ex:binom3}
  Let $X\sim \mathrm{Bin}(n, p)$.  From previous developments we
  obtain the upper and lower bounds 
\begin{align*}
&   \frac{\mathbb{E} \left[ X \Delta^-f(X) \right]^2}{np} \le
  \mathrm{Var}[f(X)] \le (1-p)\mathbb{E} \left[ X(\Delta^-f(X))^2
  \right]; \\
& \frac{\mathbb{E} \left[ (X-n) \Delta^+f(X) \right]^2}{n(1-p)} \le
  \mathrm{Var}[f(X)] 
 \le \mathbb{E} \left[ p(n-X)(\Delta^+f(X))^2
  \right].  
\end{align*}
\end{exm}
\begin{exm}[Poisson bounds]\label{ex:poiss3}
Let  $X\sim \mathrm{Pois}(\lambda)$. From previous developments we
  obtain the upper and lower bounds 
\begin{align*}
 &  \frac{\mathbb{E} \left[ X \Delta^-f(X) \right]^2}{\lambda} \le
  \mathrm{Var}[f(X)] \le \mathbb{E} \left[ X(\Delta^-f(X))^2
  \right]; \\
 &  \mathbb{E} \left[\Delta^+f(X) \right]^2 \le
\mathrm{Var}[f(X)] \le \lambda \mathbb{E} \left[  (\Delta^+f(X))^2
  \right]. 
\end{align*}
\end{exm}

\begin{exm}[Beta bounds]\label{ex:beta3}
  Let $X\sim \mathrm{Beta}(\alpha, \beta)$. From previous developments
  we obtain the upper and lower bounds
\begin{align*}
 &  \frac{(\alpha+\beta)(\alpha+\beta+1)}{\alpha\beta} \mathbb{E} \left[ X(1-X) f'(X) \right]^2 \le
  \mathrm{Var}[f(X)] \le \frac{1}{\alpha+\beta} \mathbb{E} \left[ X(1-X)(f'(X))^2
  \right]  
\end{align*}
\end{exm}

\section{A Lagrange formula and Papathanassiou-type variance
  expansions}
\label{sec:variance-expansion}

This section begins  with a third representation for the Stein operator
$\mathcal{L}_p^{\ell}$; surprisingly, this result seems to be new.
 \begin{lma}[Representation formula III]\label{lem:represnt}
   Let $X, X_1, X_2$ be independent copies of $X \sim p$ with support
   $\mathcal{S}(p)$. We define
\begin{align}
  \label{eq:29}
  \Phi^{\ell}_p(u, x, v) = \chi^{\ell}(u, x)\chi^{-\ell}(x, v)/p(x)
\end{align}
for all $u, v \in \mathcal{S}(p)$ (note that $ \Phi^{\ell}_p(u, x, v)=0$
for $u \ge v$).  Then, for all
   $h \in L^1(p)$ we have
\begin{equation}
  \label{eq:reprsntform}
  -\mathcal{L}_p^{\ell}h(x)  =  \mathbb{E} \left[
       (h(X_2)-h(X_1)) 
\Phi^{\ell}_p(X_1, x,  X_2) 
     \right]. 
\end{equation}
 \end{lma}
\begin{proof}
First notice that, under the assumptions,
  $ \mathbb{E} [\chi^\ell(X, x) + \chi^{-\ell}(x, X) ] = 1$ for all
  $x\in \mathcal{X}$. Suppose without loss of generality that
  $\mathbb{E}[h(X)]=0$. Using that $X_1, X_2$ are i.i.d., we
reap
\begin{align*}
& \mathbb{E} \big[
       (h(X_1)-h(X_2)) 
       \chi^{\ell}(X_1, x)  \chi^{-\ell}(x, X_2)
     \big]  \\
& = \mathbb{E} \big[
       h(X_1)
       \chi^{\ell}(X_1, x)
     \big]  \mathbb{E}[\chi^{-\ell}(x, X_2)]-
               \mathbb{E} \big[
       h(X_2) 
       \chi^{-\ell}(x, X_2)
     \big]  \mathbb{E}[\chi^{\ell}(X_1, x)]  \\
&   = \mathbb{E} \big[
       h(X_1)
       \chi^{\ell}(X_1, x)
     \big]  \mathbb{E}[\chi^{-\ell}(x, X_2)]-
               \mathbb{E} \big[
       h(X_2) 
       \big(1- \chi^{\ell}(X_2, x)\big)
     \big]  \mathbb{E}[\chi^{\ell}(X_1, x)] \\
  & = \mathbb{E} \big[
       h(X)
       \chi^{\ell}(X, x)
     \big]  \big(\mathbb{E}[\chi^{-\ell}(x, X)] +
    \mathbb{E}[\chi^{\ell}(X, x)]\big).  
\end{align*}
The conclusion follows after recalling \eqref{eq:17}.
\end{proof}

Combining \eqref{eq:25} and \eqref{eq:reprsntform} we easily obtain
the following identities which, together, will ultimately lead to our
infinite expansions for variances.

\begin{lma}\label{lma:var1}
  Let $X \sim p$ with support $\mathcal{S}(p)$.  For all
  $x_1, x_2 \in \mathcal{S}(p)$, we have
\begin{equation}
  \label{eq:28}
g(x_2)-g(x_1) =   \mathbb{E}\left[   \Phi^{\ell}_p(x_1, X, x_2) \Delta^{-\ell}g(X)\right].
\end{equation}
If  $X_1, X_2$ are independent  copies of $X$ then 
\begin{equation} \label{eq:27}
\mathrm{Var}[g(X)] = \E\left[ \left(g(X_2)-g(X_1) \right)^2 \mathbb{I}_{[X_1<X_2]} \right]
\end{equation}
for all  $g \in L^2(p)$.
\end{lma}
\begin{proof}
Identity \eqref{eq:28} follows from the (trivial) observation that 
\begin{align*}
\mathbb{E}\left[  \frac{\chi^{\ell}(x_1,
    X)\chi^{-\ell}(X, x_2)}{p(X)} \Delta^{-\ell}g(X)\right] & =
                                                        \int_{\mathcal{S}(p)}\chi^{\ell}(x_1,
    x)\chi^{-\ell}(x, x_2)  \Delta^{-\ell}g({\red{x}})  d\mu(x) 
\end{align*}
which is equal to $g(x_2)-g(x_1)$ for all
$x_1 < x_2 \in \mathcal{S}(p)$ and all
$\ell \in \left\{ -1, 0, 1 \right\}$. Using \eqref{eq:25} followed by 
\eqref{eq:reprsntform}, we get
\begin{align*}
  \mathrm{Var}[g(X)] & = \mathbb{E}[-\mathcal{L}_p^{\ell}g(X)
                       \Delta^{-\ell}g(X)]  \\
& = \mathbb{E} \left[
                       \mathbb{E}\big[ \big( g(X_2) - g(X_1)\big)
                      \Phi_p^{\ell}(X_1, X, X_2)\, | \, X \big]
  \Delta^{-\ell}g(X)   \right] \\
                     & = \mathbb{E} \left[ \big( g(X_2) - g(X_1)\big)
                       \Phi_p^{\ell}(X_1, X, X_2)
                       \Delta^{-\ell}g(X)   \right] \\
  & = \mathbb{E} \left[ \big( g(X_2) - g(X_1)\big) \mathbb{E} \left[  \Phi_p^{\ell}(X_1, X, X_2)
    \Delta^{-\ell}g(X) \, | \, X_1, X_2\right] \mathbb{I}_{[X_1<X_2]}
  \right] \\
  & = \mathbb{E} \left[ \big( g(X_2) - g(X_1)\big)^2
    \mathbb{I}_{[X_1<X_2]} \right],
\end{align*}
whence the claim. 
\end{proof}

\begin{rmk} 
In fact, treating the discrete and continuous cases separately, one
could also obtain identity \eqref{eq:27} as a direct application of
Lagrange's identity (a.k.a.\ the Cauchy-Schwarz inequality with remainder) which
reads, in the {{finite}} discrete case, as
\begin{equation}
  \left(\sum_{k=u}^v a_k^2 \right)\left(\sum_{k=u}^v b_k^2 \right) -
  \left(\sum_{k=u}^v a_k b_k \right)^2 =  \sum_{i=u}^{v-1}\sum_{j=i+1}^v (a_ib_j-a_jb_i)^2
\label{lagrange}
\end{equation}
with $a_k=g(k)\sqrt{p(k)}$ and $b_k=\sqrt{p(k)}$ for
$k=0,\ldots,n$. Identity \eqref{lagrange} and its continuous
counterpart will play a crucial role in the sequel; it will be more
suited to our cause under the following form.
\end{rmk}
\begin{lma}[A probabilistic Lagrange identity]
\label{lma:prob}
  Let $X, X_1, X_2$ be independent random variables with distribution
  $p$ and $a, b$ any two functions such that the expectations below
  exist.  Then
 \begin{align}
   \label{eq:lagrangeidentityp}
   \mathbb{E} \left[ a(X) b(X) \Phi^{\ell}_p(u, X, v) \right]^2  & =
                                                                \mathbb{E} \left[ (a(X))^2 \Phi^{\ell}_p(u, X, v) \right]
                                                                \mathbb{E} \left[ (b(X))^2 \Phi^{\ell}_p(u, X, v) \right] \nonumber \\
                                                              & \quad
                                                                -
                                                                \mathbb{E}
                                                                \left[
                                                                (a(X_1)b(X_2)
                                                                -
                                                                a(X_2)b(X_1))^2 
                                                                \Phi^\ell_p(u, X_1, X_2, v) \right]
 \end{align}
where 
\begin{align}
  \label{eq:51}
  \Phi^\ell_p(u, x_1, x_2, v) = \frac{\chi^{\ell}(u,
  x_1)\chi^{\ell^2}(x_1, x_2)\chi^{-\ell}(x_2, v)}{p(x_1)p(x_2)}. 
\end{align}
\end{lma}

\begin{proof}[Proof of Lemma \ref{lma:prob}]
The proof follows from expanding the second term in \eqref{eq:lagrangeidentityp}, 
\begin{eqnarray*}
\lefteqn{ \mathbb{E}    \left[
                                                                (a(X_1)b(X_2)
                                                                -
                                                                a(X_2)b(X_1))^2 
                                                                \Phi^\ell_p(u, X_1, X_2, v) \right]}\\
 &=& \mathbb{E}    \left[ ( a(X_1)^2 b(X_2)^2 + a(X_2)^2 b(X_1)^2 ) 
  \Phi^\ell_p(u, X_1, v) \Phi^\ell_p(u,  X_2, v)\Delta^{\ell^2}(X_1, X_2) \right] 
  \\
  &&- 2  \mathbb{E}    \left[ ( a(X_1) b(X_2) a(X_2) b(X_1) ) 
  \Phi^\ell_p(u, X_1, v) \Phi^\ell_p(u,  X_2, v)\Delta^{\ell^2}(X_1, X_2) \right]                                                             .
\end{eqnarray*}
By symmetry, 
\begin{eqnarray*}
\lefteqn{ \mathbb{E}    \left[
                                                                (a(X_1)b(X_2)
                                                                -
                                                                a(X_2)b(X_1))^2 
                                                                \Phi^\ell_p(u, X_1, X_2, v) \right]}\\
 &=&  \mathbb{E}    \left[
                                                                (a(X_1)b(X_2)
                                                                -
                                                                a(X_2)b(X_1))^2 
                                                                \Phi^\ell_p(u, X_2, X_1, v) \right].
 \end{eqnarray*}
Now $\Delta^{\ell^2}(X_1, X_2) + \Delta^{\ell^2}(X_2, X_1) = 1$ when $X_1 = X_2$, and vanishes otherwise, and for $X_1 = X_2$, it holds that $(a(X_1)b(X_2)
                                                                -
                                                                a(X_2)b(X_1))^2 =0$. Hence
                                                                \begin{eqnarray*}
\lefteqn{ \mathbb{E}    \left[
                                                                (a(X_1)b(X_2)
                                                                -
                                                                a(X_2)b(X_1))^2 
                                                                \Phi^\ell_p(u, X_1, X_2, v) \right]}\\
                                                                &=& \frac12 
                                                                \mathbb{E}    \left[
                                                                (a(X_1)b(X_2)
                                                                -
                                                                a(X_2)b(X_1))^2 
                                                                \Phi^\ell_p(u, X_1, v)  \Phi^\ell_p(u, X_2, v)\right].
            \end{eqnarray*}
            Exploiting the independence of $X_1$ and $X_2$ now yields
            the conclusion.                                                       
\end{proof}

\begin{rmk}
  For ease of reference, we detail \eqref{eq:51}:
\begin{align} \label{eq:30} 
& \Phi^0_p(u, x_1, x_2, v) =  \Phi^{0}_p(u,
  x_1, x_2)  \frac{\chi^0(x_2,  v)}{p(x_2)} = \frac{\chi^{0}(u,
  x_1)\chi^{0}(x_1, x_2)\chi^0(x_2, v)}{p(x_1)p(x_2)} \\
& \Phi^-_p(u, x_1, x_2, v) =  \Phi^{-}_p(u,
  x_1, x_2)  \frac{\chi^+(x_2,  v)}{p(x_2)} = \frac{\chi^{-}(u,
  x_1)\chi^{+}(x_1, x_2)\chi^+(x_2, v)}{p(x_1)p(x_2)}  \label{eq:39} \\
& \Phi^+_p(u, x_1, x_2, v) =   \frac{\chi^+(u, x_1)}{p(x_1)} \Phi^+_p(x_1,
  x_2, v)  = \frac{\chi^{+}(u,
  x_1)\chi^{+}(x_1, x_2)\chi^-(x_2, v)}{p(x_1)p(x_2)} \label{eq:42} 
\end{align}
\end{rmk}



\begin{thm}\label{thm:var-}
  Fix $\ell \in \left\{ -1, 0, 1 \right\}$ and define the sequence
  $\pmb{\ell}=(\ell_n)_{n\ge1}$ such that $\ell_n=0$ for all $n$ if $\ell =0$,
  otherwise $\ell_n\in \left\{ -1, 1 \right\}$ arbitrarily
  chosen. Consider a sequence $(h_n)_{n\ge1}$ of real valued functions
  $h_i: \R \to \R$ such that
  $\mathbb{P}[\Delta^{\ell_i} h_{2i-1}(X)>0] = 1$ for all $i\ge1$.
  Starting with some function $g$, we also recursively define the
  sequence $(g_k)_{0\leq k\leq n}$ by $g_0(x)=g(x)$ and
  $g_{i}(x)={\Delta^{\ell_i}g_i(x)}/{\Delta^{\ell_i}h_{i}(x)}$ for all
  $x \in \mathcal{S}(p)$. Finally, starting from
  $\Phi_{1,1}^{\pmb{\ell}}(x_1, x, x_2) = \Phi^{\ell_1}_p(x_1, x, x_2)$ and
  $\Phi_{1,2}^{\pmb{\ell}}(x_1,x_3, x_4, x_2) = \Phi^{\ell_1}_p(x_1, x_3, x_4,
  x_2)$ we define recursively for $n \ge 2$
\begin{align}
 &  \Phi_{n,1}^{\pmb{\ell}}(x_1, x_3, \ldots, x_{2n-1}, x, x_{2n}, \ldots,
   x_2)\nonumber\\
&  \qquad =
   \Phi^{\ell_{n}}_p(x_{2n-1}, x, x_{2n})
    \Phi_{n-1,2}^{\pmb{\ell}}(x_1, x_3, \ldots, x_{2n-1}, x_{2n}, \ldots, x_2)\label{eq:50}\\ 
 &      \Phi_{n,2}^{\pmb{\ell}}(x_1, x_3, \ldots, x_{2n-1}, x_{2n+1},
   x_{2n+2}, x_{2n}, \ldots, x_2) \nonumber \\
&  \qquad =
   \Phi^{\ell_{n}}_p(x_{2n-1}, x_{2n+1},
   x_{2n+2}, x_{2n})
    \Phi_{n-1,2}^{\pmb{\ell}}(x_1, x_3, \ldots, x_{2n-1}, x_{2n}, \ldots,
   x_2)  \label{eq:38}
\end{align}
at any sequence $(x_j)_{j\ge1}$.  Then, for all $g\in L^2(p)$ and all
$n \ge 1$ we have
\begin{align}\label{eq:varianceexpansion}
\mathrm{Var}[g(X)]=  
& \sum_{k=1}^n (-1)^{k-1} \E\left[ 
\frac{\left(\Delta^{-\ell_k} g_{k-1}(X)\right)^2}{\Delta^{-\ell_{k}}
                        h_k(X)} \Gamma_k^{\pmb \ell}(X) 
\right]		 + (-1)^n R_n^{\pmb\ell}
\end{align}
where 
\begin{align}
  &  \Gamma_k^{\pmb \ell}(x)  = \E\Bigg[ 
 (h_k(X_{2k})-h_k(X_{2k-1}))  
    \prod_{i=1}^{k-1} 
\Delta^{-\ell_i} h_i(X_{2i+1}, X_{2i+2})   
    \Phi_{k, 1}^{\pmb{\ell}}(X_1, \ldots,  X_{2k-1}, x, X_{2k}, \ldots, X_2)
    \Bigg]\label{eq:2}
\end{align}
 and 
\begin{align}
  R_n^{\pmb \ell}& =  \E \Bigg[
 \left(g_n(X_{2n+2})- g_n(X_{2n+1})\right)^2  \prod_{i=1}^n   
\Delta^{-\ell_{i}} h_i(X_{2i-1}, X_{2i})        
\Phi_{n, 2}^{\pmb{\ell}}(X_1, \ldots X_{2n+1},  X_{2n+2}, \ldots,  X_2)
         \Bigg]
\end{align}
(we introduce the notation $\Delta^{\ell}h_k(x, y) = \Delta^{\ell}h_k(x)
\Delta^{\ell}h_k(y)$ and an empty product is set to 1).
\end{thm}
 
\begin{proof}
  Throughout this proof we write
  $ \mathbb{E}_{X_{i_1}, \ldots, X_{i_k}} \left[ \Psi(X_1, \ldots,
    X_n) \right]$ to denote the expectation of multivariable
  function $\Psi(X_1, \ldots, X_n)$ taken only with respect to
  $X_{i_1}, \ldots, X_{i_k}$ (that is, conditionally on all
  non-mentioned variables).  Starting from \eqref{eq:27} and using
  \eqref{eq:28}:
  \begin{align}
    \mathrm{Var}[g(X)] & = \mathbb{E}\left[ \left(g(X_2)-g(X_1) \right)^2
                         \mathbb{I}_{[X_1<X_2]} \right] \nonumber \\
& =  \mathbb{E}\left[ \left( \mathbb{E}_X\left[   \Phi^{\ell_1}_p(X_1, X, X_2)
  \Delta^{-\ell_1}g(X)\right]\right)^2
                         \mathbb{I}_{[X_1<X_2]} \right]. 
\label{eq:32}  \end{align}
Next, for any $h_1$ such that $\mathbb{P}[\Delta^{-\ell_1}h_1(X)>0]=1$, we
have, thanks to \eqref{eq:lagrangeidentityp} and conditionally on
$X_1, X_2$: 
\begin{align} 
  &  \bigg(\mathbb{E}_X\left[   \Phi^{\ell_1}_p(X_1, X, X_2)
  \Delta^{-\ell_1}g(X)\right]\bigg)^2   =  \bigg(\mathbb{E}_X\left[   \Phi^{\ell_1}_p(X_1, X, X_2)
  \frac{\Delta^{-\ell_1}g(X)}{\sqrt{\Delta^{-\ell_1}h_1(X)}}\sqrt{\Delta^{-\ell_1}h_1(X)}\right]\bigg)^2
    \nonumber \\
  &   = \mathbb{E}_X\left[   
  \frac{\big(\Delta^{-\ell_1}g(X)\big)^2}{{\Delta^{-\ell_1}h_1(X)}} \Phi^{\ell_1}_p(X_1, X, X_2) \right]
    \mathbb{E}_{X'}\left[  
    \Delta^{-\ell_1}h(X')  \Phi^{\ell_1}_p(X_1, X', X_2) \right]\nonumber   \\
  & \quad  - \mathbb{E}_{X_3, X_4}\left[  
  \bigg(
    \frac{\Delta^{-\ell_1}g(X_3)}{\sqrt{\Delta^{-\ell_1}h_1(X_3)}}\sqrt{\Delta^{-\ell_1}h_1(X_4)}
    - \frac{\Delta^{-\ell_1}g(X_4)}{\sqrt{\Delta^{-\ell_1}h_1(X_4)}}\sqrt{\Delta^{-\ell_1}h_1(X_3)}
    \bigg)^2 
    \Phi^{\ell_1}_p(X_1, X_3, X_4, X_2) \right] \nonumber  \\
  & =: I_1(X_1, X_2) - I_2(X_1, X_2)\label{eq:31}
\end{align}
($X'$ in the second line is another independent copy of $X$). 
We need to compute
$ \mathbb{E}[I_1(X_1, X_2) \mathbb{I}_{[X_1<X_2]}] -
\mathbb{E}[I_2(X_1, X_2) \mathbb{I}_{[X_1<X_2]}]$. We begin by noting
that, in the discrete case, the strict inequality in the indicator
$\mathbb{I}_{[X_1<X_2]}$ is implicit in
$\Phi^{\ell_1}_p(X_1, X, X_2) = \chi^{\ell_1}(X_1, X)\chi^{-\ell_1}(X,
X_2)/p(X)$ and hence a fortiori also in
$\Phi^{\ell_1}_p(X_1, X_3, X_4, X_2)$. In the continuous case there is
no difference between $\mathbb{I}_{[X_1<X_2]}$ and
$\mathbb{I}_{[X_1\le X_2]}$.  We treat the two terms
separately. First,
\begin{align}
& \mathbb{E}[I_1(X_1, X_{2})\mathbb{I}_{[X_1<X_2]}]\nonumber\\
 & = \mathbb{E}\left[ \mathbb{E}_X\left[     
  \frac{\big(\Delta^{-\ell_1}g(X)\big)^2}{{\Delta^{-\ell_1}h_1(X)}}
                  \Phi^{\ell_1}_p(X_1, X, X_2) \right]
   \mathbb{E}_{X'}\left[  
    \Delta^{-\ell_1}h_1({X'})  \Phi^{\ell_1}_p(X_1, {X'}, X_2) \right] \mathbb{I}_{[X_1<X_2]}  \right] \nonumber\\
  & =  \mathbb{E} \left[
    \frac{\big(\Delta^{-\ell_1}g(X)\big)^2}{{\Delta^{-\ell_1}h_1(X)}}
    \Phi^{\ell_1}_p(X_1, X, X_2) \mathbb{E}_{X'}\left[  
    \Delta^{-\ell_1}h_1({X'})  \Phi^{\ell_1}_p(X_1, {X'}, X_2) \right]  \right] \nonumber\\
  & =  \mathbb{E} \left[
    \frac{\big(\Delta^{-\ell_1}g(X)\big)^2}{{\Delta^{-\ell_1}h_1(X)}}
    \big(h_1(X_2)-h_1(X_1)\big) \Phi^{\ell_1}_p(X_1, X, X_2)\right], \label{eq:46}
\end{align}
the last identity by \eqref{eq:28}. For the second term we have
\begin{align}
&   \mathbb{E}[I_2(X_1, X_{2})\mathbb{I}_{[X_1<X_2]}] \nonumber\\
&  = \mathbb{E} \left[ 
                    \bigg(
                    \frac{\Delta^{-\ell_1}g(X_3)}{\sqrt{\Delta^{-\ell_1}h_1(X_3)}}\sqrt{\Delta^{-\ell_1}h_1(X_4)}
                    - \frac{\Delta^{-\ell_1}g(X_4)}{\sqrt{\Delta^{-\ell_1}h_1(X_4)}}\sqrt{\Delta^{-\ell_1}h_1(X_3)}
                    \bigg)^2 
                    \Phi^{\ell_1}_p(X_1, X_3, X_4, X_2) \right] \nonumber\\
  & = \mathbb{E} \left[ \Delta^{-\ell_1}h_1(X_3) \Delta^{-\ell_1}h_1(X_4) 
                    \big( g_1(X_4) - g_1(X_3)                  
                    \big)^2 
                    \Phi^{\ell_1}_p(X_1, X_3, X_4, X_2) \right] \label{eq:49}
\end{align}
with $g_1(x) =
{\Delta^{-\ell_1}g(x)}/{\Delta^{-\ell_1}h_1(x)}$, as
anticipated. Combining \eqref{eq:46} and \eqref{eq:49} we obtain
\eqref{eq:varianceexpansion} at $n=1$.  

 To pursue towards $n=2$, starting
from \eqref{eq:49}, we simply apply the same process as above:
\begin{align*}
&  \mathbb{E}
                 \left[  \Delta^{-\ell_1}h_1(X_3, X_4)  \left(
                 g_1(X_4) - g_1(X_3) \right)^2
                 \Phi^{\pmb\ell}_{1,2}(X_1, X_3, X_4, X_2)\right] \\
  & = \mathbb{E} \left[ \Delta^{-\ell_1}h_1(X_3, X_4)
    \left(\mathbb{E}_X\left[   \Phi^{\ell_2}_p(X_3, X, X_4) 
  \Delta^{-\ell_2}g_1(X)\right]\right)^2\Phi^{\pmb\ell}_{1,2}(X_1, X_3, X_4, X_2) \right] \\
& =: \mathbb{E} \left[ \Delta^{-\ell_1}h_1(X_3, X_4)   I_{21}(X_3, X_4)
 \Phi^{\pmb\ell}_{1,2}(X_1, X_3, X_4, X_2) \right] \\
& \quad - \mathbb{E} \left[ \Delta^{-\ell_1}h_1(X_3, X_4)   I_{22}(X_3, X_4)
 \Phi^{\pmb\ell}_{1,2}(X_1, X_3, X_4, X_2)  \right] 
\end{align*}
with $I_{2j}, j=1, 2$ defined (in terms of a function $h_2$) through
\eqref{eq:31}. Again we treat the terms separately: 
\begin{align*}
  &  \mathbb{E} \left[ \Delta^{-\ell_1}h_1(X_3, X_4)   I_{21}(X_3, X_4)
   \Phi^{\pmb\ell}_{1,2}(X_1, X_3, X_4, X_2)  \right]\\
  & =   \mathbb{E} \left[
    \frac{(\Delta^{-\ell_2}g_1(X))^2}{\Delta^{-\ell_2}h_2(X)}\big(h_2(X_4)-h_2(X_3)\big)  
    \Delta^{-\ell_1}h_1(X_3, X_4) \Phi^{\ell_2}_p(X_3, X, X_4)   \Phi^{\pmb\ell}_{1,2}(X_1, X_3, X_4, X_2)\right] 
\end{align*}
and 
\begin{align*}
  &   \mathbb{E} \left[ \Delta^{-\ell_1}h_1(X_3, X_4)    I_{22}(X_3, X_4)
    \Phi^{\pmb\ell}_{1,2}(X_1, X_3, X_4, X_2)  \right]  \\
  & = \mathbb{E} \left[ \Delta^{-\ell_1}h_1(X_3, X_4)
    \Delta^{-\ell_2}h_2(X_5, X_6) \big(g_2(X_6) -
    g_2(X_5)\big)^2 \Phi^{\ell_2}_p(X_3, X_5, X_6, X_4)
    \Phi^{\pmb\ell}_{1,2}(X_1, X_3, X_4, X_2)\right].     
\end{align*}
  The
complete statement follows by induction and careful bookkeeping of all
the indices.
\end{proof}

\begin{rmk}[About the assumptions in the theorem]
  A stronger sufficient condition on the functions $h_i$ is that they
  be strictly increasing throughout $\mathcal{S}(p)$, in which case
  the condition $\Delta^{\ell}h(x)>0$ is guaranteed. The prototypical
  example of such a sequence is $h_i(x) = x$ for all $i \ge 1$ which
  clearly satisfies all the required assumptions.  When $\ell\neq 0$
  then the condition that $\mathbb{P}[\Delta^{\ell_k}h_k(X)>0]=1$ is
  in fact too restrictive because, as has been made clear in the
  proof, the recurrence only implies that $\Delta^{\ell_k}h_k(x)$
  needs to be positive on $(a+i(k), b-j(k))$ for some indices
  $i(k), j(k)$ depending on the behavior of $X_{2k-1}$ and $X_{2k}$
  (see the line just below equation \eqref{eq:32} when $k=1$). In
  particular the sequence necessarily stops if $\mathcal{S}(p)$ is
  bounded, since after a certain number of iterations the indicator
  functions defining $\Phi_{n, j}^{\pmb\ell}$ will be 0 everywhere. 
\end{rmk}

\begin{rmk}
  When $g$ is a $d$th-degree polynomial, we obtain an exact expansion
  of the variance in \eqref{eq:varianceexpansion} with respect to the
  $\Gamma_k^{\pmb \ell}(x)$ functions ($k=1,\ldots,d$) as the
  $R_{n}^{\pmb \ell}$ is defined in terms of the $n$-th derivative of
  $g$.
\end{rmk}

The functions $\Gamma_k^{\ell_k}$ defined in \eqref{eq:2} are some
sort of generalization of Peccati's and Ledoux' gamma mentioned in the
Introduction. To see this, choose $h_k(x) = x$ for all $k$ (arguably
the most natural choice) for which
$ \Delta^{\ell}h_k(x) = 1 \mbox{ for all } k$ so that
  \begin{align}\label{eq:55}
  \Gamma_{k}^{\pmb \ell}(x) 
  & = \mathbb{E} \left[ (X_{2k}-X_{2k-1})\Phi_{k,
    1}^{\pmb\ell}(X_1, \ldots,   X_{2k-1}, x, X_{2k}, \ldots,  X_2)
    \right]
  \end{align}
  for all $k\ge1$ we have.  Expliciting the above leads to the
  expressions:
\begin{align*}
&   \Gamma_1^{\ell_1}(x)  = \mathbb{E} \bigg[(X_2-X_1)
  \frac{\chi^{\ell_1}(X_1, x) \chi^{-\ell_1}(x, X_2)}{p(x)}\bigg] \\
  &   \Gamma_2^{\ell_1, \ell_2}(x)  = \mathbb{E} \bigg[(X_4-X_3)
  \frac{\chi^{\ell_1}(X_1, X_3) \chi^{\ell_2}(X_3, x)\chi^{-\ell_2}(x,
    X_4)\chi^{-\ell_1}(X_4, X_2)}{p(X_3)p(x)p(X_4)}\bigg]\\
  &   \Gamma_3^{\ell_1, \ell_2, \ell_3}(x)  = \mathbb{E} \bigg[(X_6-X_5)
  \frac{\chi^{\ell_1}(X_1, X_3) \chi^{\ell_2}(X_3,
    X_5)\chi^{\ell_3}(X_5,x)
    \chi^{-\ell_3}(x,
    X_6)\chi^{-\ell_2}(X_4, X_2)\chi^{-\ell_1}(X_4,
    X_2)}{p(X_3)p(X_5)p(x)p(X_6)p(X_4)}\bigg] \\
  & \mathrm{etc.}
\end{align*}
In particular $\Gamma_1^{\pmb \ell}(x) = \tau_p^{\ell_1}(x)$ is the
Stein kernel of $X$. 
The next lemma follows {by induction.}
\begin{lma}\label{lma:gamma}
If $\ell = 0$ then
\begin{equation}
\Gamma_{k}^{0}(x)  = \label{eq:57}
    \frac{1}{p(x)}\mathbb{E}
    \Bigg[\frac{1}{k!(k-1)!} (x-X_{1})^{k-1}(X_{2}-x)^{k-1}(X_{{2}}-X_{1})  \mathbb{I}_{[X_1
      \le x \le X_2]} \Bigg]. 
\end{equation}
{
and if ${\pmb \ell} \in \{-1,1\}^k$, then 
\begin{equation}
\Gamma_{k}^{\pmb \ell}(x)  = \label{eq:57Discrete}
   \frac{1}{p(x)} \mathbb{E}
    \Bigg[ \frac{1}{k!(k-1)!} (x-X_{1}-a_{\pmb \ell} +1)^{\lceil k-1\rceil}(X_{2}-x-b_{\pmb \ell} +1)^{\lceil k-1\rceil}(X_{{2}}-X_{1})  \mathbb{I}_{[X_1 + a_{\pmb \ell}  \le x \le X_2 + b_{\pmb \ell}]} \Bigg]. 
\end{equation}
where 
$a^{\lceil k\rceil}(x)= a(a+1)\cdots (a+k-1)$, $a_{\pmb \ell}= \sum_{i=1}^k \frac{l_i(l_i+1)}{2}$ and $b_{\pmb \ell}= \sum_{i=1}^k \frac{l_i(1-l_i)}{2}$. 
}
\end{lma}

\begin{proof}

We define $\gamma_{k}^{\pmb \ell}(x,X_1,X_2)$ such that 
\begin{equation}
\Gamma_k^{\pmb \ell}(x) = \frac{1}{p(x)} \mathbb{E} \bigg[ \gamma_k^{\pmb \ell}(x,X_1,X_2) \bigg]
\end{equation}
The proof is complete if we have, in the continuous case,
\begin{equation}\label{lmaCont}
\gamma_k^0(x) = \frac{1}{k!(k-1)!} (x-X_{1})^{k-1}(X_{2}-x)^{k-1}(X_{{2}}-X_{1}) \mathbb{I}_{[X_1 \le x \le X_2 ]}
\end{equation}
and in the discrete case,
\begin{equation}\label{lmaDisc}
\gamma_k^{\pmb \ell}(x) = \frac{1}{k!(k-1)!} (x-X_{1}-A_{\pmb \ell} +1)^{\lceil k-1\rceil}(X_{2}-x-B_{\pmb \ell} +1)^{\lceil k-1\rceil}(X_{{2}}-X_{1})  \mathbb{I}_{[X_1 + A_{\pmb \ell}  \le x \le X_2 + B_{\pmb \ell}]}
\end{equation}

We prove \eqref{lmaCont} and \eqref{lmaDisc} by induction. 
Firstly, for $k=1$, $\ell=0,1$ or -1, we obtain
\begin{align*}
  \Gamma^\ell_{1}(x) & = \frac{1}{p(x)}\mathbb{E} \Bigg[(X_{{2}}-X_{1})  \mathbb{I}_{[X_1+\frac{\ell(\ell+1)}{2} \le x \le  X_2 + \frac{\ell(1-\ell)}{2}]}\Bigg]\\
\end{align*}

After conditioning with respect to the ``extreme'' variables, the expression of \eqref{eq:55} can be rewritten as  
\begin{align}
\Gamma_{k+1}^{\pmb \ell}(x) 
=& \mathbb{E} \left[ (X_{2k+2}-X_{2k+1})\Phi_{k+1,1}^{\pmb\ell}(X_1, \ldots,   X_{2k+1}, x, X_{2k+2}, \ldots, X_2) \right] \nonumber\\
=& \mathbb{E} \bigg[ (X_{2k+2}-X_{2k+1}) 
\prod_{i=1}^{k} \left( \frac{\chi^{\ell_i}(X_{2i-1},X_{2i+1})}{p(X_{2i+1})} \frac{\chi^{-\ell_i}(X_{2i+2},X_{2i})}{p(X_{2i+2})} \right) \frac{\chi^{\ell_{k+1}}(X_{2k+1},x) \chi^{-\ell_{k+1}}(x,X_{2k+2})}{p(x)}
\bigg] \nonumber\\
%
=:& \frac{1}{p(x)} \mathbb{E}_{X_1,X_2} \bigg[
\mathbb{E}_{X_3,X_4} \bigg[
\frac{\chi^{\ell_1}(X_{1},X_{3})}{p(X_{3}} \frac{\chi^{-\ell_1}(X_{4},X_{2})}{p(X_{4}}
\gamma_k^{\ell_2,\ldots,\ell_{k+1}}(x,X_3,X_4)
 \, \Bigg|\, X_1, X_2  \bigg] \bigg]
\end{align}
This expression allows us to conclude the assertion by induction,  in the continuous case and in the discrete case separately.
\end{proof}

\begin{rmk}
  The results e.g.\ from \cite{papathanasiou1988variance,J93,APP07,APP11,AfBalPa2014} permit to
  identify that if $X \sim p$ is Pearson distributed then
  $\Gamma_{k}^0(x) = \tau_p(x)^k$, and if $X$ is Ord distributed then
  $\Gamma_{k}^0(x) = \tau_p(x)^{[k]}$ with
  $f^{[k]}(x) = f(x-k+1) \cdots f(x)$. 
\end{rmk}

\begin{exm}[Normal bounds] \label{ex:normbounds}
  Direct computations show that if $X \sim \mathcal{N}(0, 1)$ then
  $\Gamma_1^0(x) = 1$ and {$\Gamma_2^0(x) = \frac{1}{2}$} so that the first two  bounds become
  \begin{align*}
    \mathrm{Var}[g(X)] &  = \mathbb{E} \left[ g'(X)^2 \right] - R_1  \\
 &  = \mathbb{E} \left[ g'(X)^2 \right] -
    \frac{1}{2}\mathbb{E} \left[ g''(X)^2 \right] + R_2
  \end{align*}
\end{exm}

\begin{exm}[Beta bounds] \label{ex:betabounds}
  Direct computations show that if $X \sim \mathrm{Beta}(\alpha, \beta)$ then
  $\Gamma_1^0(x)=x(1-x)/(\alpha+\beta)$ and {$\Gamma_2^0(x) =
  (x(1-x))^2/(2(\alpha+\beta)(\alpha+\beta+1))$} so that the first two bounds become
  \begin{align*}
    \mathrm{Var}[g(X)] &  = \frac{1}{\alpha+\beta}\mathbb{E} \left[
                         X(1-X)g'(X)^2 \right]
                         - R_1\\
 &  =  \frac{1}{\alpha+\beta} \mathbb{E} \left[X(1-X)  g'(X)^2 \right] -
   \frac{1}{2(\alpha+\beta)(\alpha+\beta+1)}  \mathbb{E} \left[ {X^2(1-X)^2}
   g''(X)^2 \right] +  R_2
  \end{align*}
  
\end{exm}
\begin{exm}[Gamma bounds]
  \label{ex:gammabounds}
  Direct
  computations show that $X \sim \mathrm{Gamma}(\alpha, \beta)$ then
  $\Gamma_1^0(x) = \beta x$ and
  $\Gamma_2^0(x) = \frac{1}{2}\beta^2x^{2}$ so that the
  first two bounds become
  \begin{align*}
    \mathrm{Var}[g(X)] &  =  \beta \mathbb{E} \left[
                         X g'(X)^2 \right]
                         - R_1\\
 &  =  \beta \mathbb{E} \left[ X  g'(X)^2 \right] -
    \frac{\beta^2}{2}\mathbb{E} \left[   X^2 
   g''(X)^2 \right] +  R_2
  \end{align*}
\end{exm}
\begin{exm}[Laplace bounds] \label{ex:laplacebounds} Direct
  computations show that $X \sim \mathrm{Laplace}(0, 1)$ then
{$\Gamma_1^0(x) = 1+|x|$ and $\Gamma_2^0(x) = \frac{1}{2}x^2 + |x| + 1$} so that the
  first two bounds become
  \begin{align*}
    \mathrm{Var}[g(X)] &  =  
   { \mathbb{E} \left[(1+|X|)g'(X)^2 \right] - R_1 }\\
& =  \mathbb{E} \left[(1+|X|) g'(X)^2 \right] -
    \mathbb{E} \left[ (1+|X|+ X^2/2) g''(X)^2 \right] +  R_2
  \end{align*}
\end{exm}
{In the discrete case}, there is much more flexibility in the
construction of the bounds as any permutation of $+1$ and $-1$ is
allowed at every stage (that is, 
for every $k$), leading to: 
\begin{align*}
    \mathrm{Var}[g(X)] &  =  \mathbb{E} \left[
                         \Gamma_1^+(X)    (\Delta^{-}g(X))^2 \right] -
                         R_1^{+}  = \mathbb{E} \left[
                         \Gamma_1^-(X)    (\Delta^{+}g(X))^2 \right] -
                         R_1^-
  \end{align*}
  and at order 2:
{  \begin{align*}
 \mathrm{Var}[g(X)] &  =  
 \mathbb{E} \left[ \Gamma_1^+(X)(\Delta^{-}g(X))^2 \right] -
 \mathbb{E} \left[ \Gamma_2^{++}(X) (\Delta^{-, -}g(X))^2 \right]  + 
                         R_2^{++}\\
 &  =  \mathbb{E} \left[ \Gamma_1^+(X) (\Delta^{-}g(X))^2 \right] -
                         \mathbb{E} \left[\Gamma_2^{+-}(X) (\Delta^{-, +}g(X))^2
                         \right]  +  R_2^{+-}\\
   &  =  \mathbb{E} \left[
                         \Gamma_1^-(X)    (\Delta^{+}g(X))^2 \right] -
                         \mathbb{E} \left[ \Gamma_2^{-+}(X) (\Delta^{+,-}g(X))^2
                         \right]  +  R_2^{-+}\\
       &  =  \mathbb{E} \left[ \Gamma_1^-(X) (\Delta^{+}g(X))^2 \right] -
                         \mathbb{E} \left[
                         \Gamma_2^{--}(X) (\Delta^{+, +}g(X))^2 \right]  +  R_2^{--}. 
\end{align*}
}{where we use the concise notation $\Delta^{\ell_1, \ell_2}g(X)$ for $\Delta^{\ell_2} \left( \Delta^{\ell_1}g(X)\right)$.}

  We detail the bounds in several settings.
\begin{exm}[Binomial bounds]\label{ex:binomgamma}
Direct
  computations show that $X \sim \mathrm{Bin}(n, p)$ then
  \begin{equation*}
    \Gamma_1^+(x) = (1-p) x, \quad   \Gamma_1^-(x) =  p(n-x)
  \end{equation*}
    and
    \begin{equation*}
      \Gamma_2^{++}(x) = \frac{1}{2}(1-p)^2x(x-1),
      \quad  \Gamma_2^{+-}(x) =  \Gamma_2^{-+}(x) =\frac{1}{2}p (1-p) x(n-x) 
\mbox{ and }  \Gamma_2^{--}(x) = \frac{1}{2}p^2(n-x-1)(n-x) 
    \end{equation*}
    so that the first bounds become at order 1:
  \begin{align*}
    \mathrm{Var}[g(X)] &  =  (1-p) \mathbb{E} \left[X (\Delta^{-}g(X))^2 \right] - R_1^{+}\\
    & = p \mathbb{E} \left[(n-X)  (\Delta^{+}g(X))^2 \right] -  R_1^-
  \end{align*}
  and at order 2:
{  \begin{align*}
\mathrm{Var}[g(X)] 
&  =  (1-p) \mathbb{E} \left[ X  (\Delta^{-}g(X))^2 \right] - 
\frac{1}{2}(1-p)^2 \mathbb{E} \left[ X(X-1)  (\Delta^{-, -}g(X))^2 \right]  +  R_2^{++}\\
&  =  (1-p) \mathbb{E} \left[ X  (\Delta^{-}g(X))^2 \right] -
\frac{1}{2}p (1-p)  \mathbb{E} \left[X(n-X) (\Delta^{-,+}g(X))^2 \right]  +  R_2^{+-}\\
&  = p  \mathbb{E} \left[(n-X)   (\Delta^{+}g(X))^2 \right] -
\frac{1}{2}p (1-p) \mathbb{E} \left[ X(n-X) (\Delta^{+,-}g(X))^2 \right]  + R_2^{-+} \\
&  = p\mathbb{E} \left[ (n-X)   (\Delta^{+}g(X))^2 \right] -
 \frac{1}{2}p^2 \mathbb{E} \left[ (n-X-1)(n-X) (\Delta^{+, +}g(X))^2 \right]  + R_2^{--}\\
\end{align*}
}
  We note that \cite[Theorem 1.3]{hillion2011natural} prove the bound
  \begin{equation*}
    \mbox{Var}[g(X)] \le \mathbb{E} \left[ \left( \frac{X}{n}
        \Delta^-g(X) + \frac{n-X}{n} \Delta^+g(X) \right)^2 \right]
  \end{equation*}
  which is very close to a combination of the above (see their Remark
  3.3) but, as far as we can see, remains a different result. The
  similarity of the two results is striking but seemingly fortuitous.
\end{exm}

\begin{exm}[Poisson bounds]\label{ex:poissgamma}
Direct
  computations show that if $X \sim \mathrm{Pois}(\lambda)$ then
  \begin{equation*}
    \Gamma_1^+(x) = x, \quad   \Gamma_1^-(x) =  \lambda
  \end{equation*}
    and
    \begin{equation*}
      \Gamma_2^{++}(x) = \frac{1}{2}x(x-1),
      \quad  \Gamma_2^{+-}(x) =  \Gamma_2^{-+}(x) =\frac{1}{2} \lambda
      x 
\mbox{ and }  \Gamma_2^{--}(x) = \frac{1}{2}\lambda^2
    \end{equation*}
    so that the first bounds become at order 1:
    \begin{align*}
    \mathrm{Var}[g(X)] &  =  \mathbb{E} \left[
X    (\Delta^{-}g(X))^2 \right] - R_1^{+}\\
    & = \lambda \mathbb{E} \left[(\Delta^{+}g(X))^2 \right] -  R_1^-
  \end{align*}
  and at order 2:
{  \begin{align*}
\mathrm{Var}[g(X)] 
&  =  \mathbb{E} \left[ X  (\Delta^{-}g(X))^2 \right] - 
\frac{1}{2} \mathbb{E} \left[ X(X-1)  (\Delta^{-, -}g(X))^2 \right]  +  R_2^{++}\\
&  =  \mathbb{E} \left[ X  (\Delta^{-}g(X))^2 \right] -
\frac{1}{2}\lambda  \mathbb{E} \left[X (\Delta^{-,+}g(X))^2 \right]  +  R_2^{+-}\\
&  = \lambda  \mathbb{E} \left[(\Delta^{+}g(X))^2 \right] -
\frac{1}{2} \lambda \mathbb{E} \left[ X (\Delta^{+,-}g(X))^2 \right]  + R_2^{-+} \\
&  = \lambda \mathbb{E} \left[ (\Delta^{+}g(X))^2 \right] -
 \frac{1}{2} \lambda^2 \mathbb{E} \left[ (\Delta^{+, +}g(X))^2 \right]  + R_2^{--}
\end{align*}
}
\end{exm}

\section{Olkin-Shepp-type bounds}
\label{sec:mult-extens-olkin}

As mentioned in the introduction, a matrix extension of Chernoff's
gaussian bound \eqref{eq:70} is due to \cite{olkin2005matrix}, and the
result is obtained by expanding the test functions in the Hermite
basis. An extension of this result to a wide class of multivariate
densities is proposed in \cite{afendras2011matrix} (see also
references therein). Once again, our notations allow to extend this
result to arbitrary densities of real-valued random variables.
\begin{thm}[Olkin-Shepp first order bound]\label{theo:oklinshepp}
  Let all previous notations  and assumptions prevail. Then, for all
  $h$ such that $\Delta^{-\ell}h(X)\neq 0$ a.s.\ 
  \begin{align}\nonumber
  &   \begin{pmatrix}
      \mathrm{Var}[f(X)] & \mathrm{Cov}[f(X), g(X)] \\
      \mathrm{Cov}[f(X), g(X)] & \mathrm{Var}[g(X)]
    \end{pmatrix} 
    \le  \label{eq:74}
\mathbb{E} \left[
 \begin{pmatrix}\big(\Delta^{-\ell}f(X)\big)^2 &
\Delta^{-\ell}f(X)\Delta^{-\ell}g(X)\\
\Delta^{-\ell}f(X)\Delta^{-\ell}g(X) & \big(\Delta^{-\ell}g(X)\big)^2
\end{pmatrix}\frac{\Gamma_1^{\ell}(X)}{\Delta^{-\ell}h(X)} 
\right]
  \end{align}
   with
 \begin{equation*}
   \Gamma_1^{\ell}(x) =  \mathbb{E} \left[  (h(X_2)-h(X_1))
     \Phi_p^{\ell}(X_1, x, X_2)\right],
 \end{equation*}
as defined in \eqref{eq:2}. 
 \end{thm}

\begin{rmk}
Taking determinants in Theorem \ref{theo:oklinshepp} 
gives the variance-covariance inequality
 \begin{align}\nonumber
 & \mathrm{Var}[f(X)]  \mathrm{Var}[g(X)] - \{  \mathrm{Cov}[f(X), g(X)] \}^2    \\ 
   \label{eq:72}
   & \le    \mathbb{E}
     \left[\big(\Delta^{-\ell}f(X)\big)^2
     \frac{\Gamma_1^{\ell}(X)}{\Delta^{-\ell}h(X)} \right]
     \mathbb{E} \left[
     \big(\Delta^{-\ell}g(X)\big)^2\frac{\Gamma_1^{\ell}(X)}{\Delta^{-\ell}h(X)}
     \right]  -
     \mathbb{E} \left[ 
\Delta^{-\ell}f(X)\Delta^{-\ell}g(X)  
     \frac{\Gamma_1^{\ell}(X)}{\Delta^{-\ell}h(X)}
     \right]^2.  
 \end{align}
 This observation was made in \cite{olkin2005matrix} for the special
 case of $p$ being the normal distribution
 $\mathcal{N}(\mu, \sigma^2)$ and $h(x)=x$,  for which
 $\Delta^{-\ell}$ is the usual derivative and
 $\Gamma_1^{\ell}(x) = \sigma^2$.
\end{rmk}

\begin{rmk}
Multiplying the result in Theorem \ref{theo:oklinshepp} by the vector $(1,0)'$
gives the variance inequality
 \begin{align}\label{rem:2}
 \mathrm{Var}[f(X)]   \le    \mathbb{E}
     \left[\big(\Delta^{-\ell}f(X)\big)^2
     \frac{\Gamma_1^{\ell}(X)}{\Delta^{-\ell}h(X)} \right]
 \end{align}
 which corresponds to the first order of Theorem \ref{thm:var-}.
\end{rmk}

The proof of Theorem \ref{theo:oklinshepp} relies on a lemma which
seems  natural but for which we have not found a reference,
and therefore we include one for completeness.

\begin{lma}[Matrix Cauchy-Schwarz inequality]\label{lma:matrixcauchyschwarz}
  Let $\mathbf{v}(x) = (a(x), b(x))' \in \R^2$ and $f: \R \to \R$ any
  function. Then
  \begin{align}
    & \mathbb{E} \left[ \mathbf{v}(X) f(X) \Phi_p{^\ell}(u,X, v)  \right]
    \mathbb{E} \left[ \mathbf{v}(X)' f(X) \Phi_p{^\ell}(u,X, v)  \right]
      \nonumber \\
&    \qquad \qquad  \le     \mathbb{E} \left[ \mathbf{v}(X)\mathbf{v}(X)' 
      \Phi_p{^\ell}(u,X, v)  \right]  \mathbb{E} \left[ f(X)^2 \Phi_p{^\ell}(u, X,
                     v) \right], \label{eq:75}                    
  \end{align}
  where inequality in the above  is in  the  Loewner ordering, i.e.\
  the difference is nonnegative definite. 
\end{lma}
\begin{rmk}
  In light of the proof of \eqref{eq:75}, there is no doubt that a similar matrix
  version of the Lagrange identity \eqref{eq:lagrangeidentityp} holds
  as well. We have not
  explored this any further, but in particular  higher order upper bound as in \cite{afendras2011matrix} should be within reach.
\end{rmk}

\begin{proof}[Proof of Lemma \ref{lma:matrixcauchyschwarz}]
  In the proof we use the shorthand $h$ instead of $h(X)$, $h_i$
  instead of $h(X_i)$, $\Phi_p$ {or $\Phi_p^i$ for } $\Phi{_p^\ell}(u, X{_i}, v)$ and
  $\Phi_p^{ij} = \Phi{_p^\ell}(u, X_i, X_j, v)$. Writing \eqref{eq:75} out in
  coordinates, it is necessary to prove that
  \begin{align*}
   R :=     \begin{pmatrix}
  \mathbb{E} \left[ a^2 \Phi_p \right]      \mathbb{E} \left[ f^2
        \Phi_p \right] -       \mathbb{E} \left[ a f  \Phi_p \right]^2
      & \mathbb{E}[ab\Phi_p ]  \mathbb{E} \left[ f^2
        \Phi_p \right]  -   \mathbb{E} \left[ a f
        \Phi_p \right] \mathbb{E} \left[{b}  f
        \Phi_p \right] \\
\mathbb{E}[ab\Phi_p ]  \mathbb{E} \left[ f^2
        \Phi_p \right]  -   \mathbb{E} \left[ a f
        \Phi_p \right] \mathbb{E} \left[{b}  f
        \Phi_p \right] & 
      \mathbb{E} \left[ b^2 \Phi_p \right]      \mathbb{E} \left[ f^2
        \Phi_p \right] -       \mathbb{E} \left[ b f  \Phi_p \right]^2
    \end{pmatrix}
  \end{align*}
{is nonnegative definite.}  
  Next we apply \eqref{eq:lagrangeidentityp} to each
  component of the matrix $R$, which leads to
  \begin{equation*}
    R =
    \begin{pmatrix}
  \mathbb{E} \left[ (a_3 f_4 - a_4 f_3)^2 \Phi_p^{34} \right]       
      & \mathbb{E} \left[ (a_3 f_4 - a_4 f_3) (b_3 f_4 - b_4 f_3) \Phi_p^{34} \right]        \\
\mathbb{E} \left[ (a_3 f_4 - a_4 f_3) (b_3 f_4 - b_4 f_3) \Phi_p^{34} \right] & 
       \mathbb{E} \left[ (b_3 f_4 - b_4 f_3)^2 \Phi_p^{34} \right]   
    \end{pmatrix}. 
  \end{equation*}
Clearly the diagonal terms are already of the correct form. For the
off-diagonal terms we note that
\begin{align*}
  &  \mathbb{E}[( a_3f_4 -a_4f_3)( b_3f_4 -b_4f_3) \Phi_p^{34}]  \\
  & =
\mathbb{E}[a_3f_4b_3f_4 \Phi_p^{34}] - \mathbb{E}[a_3f_4b_4f_3\Phi_p^{34}] -
                                                                \mathbb{E}[a_4
                                                                f_3
                                                                b_3f_4\Phi_p^{34}]
                                                                +
                                                                \mathbb{E}[a_4f_3b_{4}
                                                                f_3\Phi_p^{34}]
\end{align*}
so that, using 
\begin{align*}
  \Phi_p^{34} = \Phi_p(u, X_3, X_4, v) =\Phi_p(u, X_3, v) \Phi_p(u,
  X_4, v)  \mathbb{I}[X_3 < X_4],
\end{align*}
we get
\begin{align*}
   &  \mathbb{E}[( a_3f_4 -a_4f_3)( b_3f_4 -b_4f_3) \Phi_p^{34}]  \\
  & =
\mathbb{E}\left[a_3b_3\Phi_p^3 f_4^2 \Phi_p^4 \mathbb{I}[X_3 < X_4] \right]
   - \mathbb{E}\left[a_3f_3 \Phi_p^3 b_4f_4  \Phi_p^4 \mathbb{I}[X_3
    < X_4] \right]
  \\
  & \quad -\mathbb{E}\left[ b_3f_3 \Phi_p^3  a_4f_4\Phi_p^4 \mathbb{I}[X_3 < X_4] \right]
  +  \mathbb{E}\left[f_3^2 \Phi_p^3 a_4b_4  \Phi_p^4 \mathbb{I}[X_3
    < X_4] \right] \\
  & = \mathbb{E}\left[a_3b_3\Phi_p^3 f_4^2 \Phi_p^4 \mathbb{I}[X_3 < X_4] \right]
   - \mathbb{E}\left[a_3f_3 \Phi_p^3 b_4f_4  \Phi_p^4 \mathbb{I}[X_3
    < X_4] \right]
  \\
  & \quad -\mathbb{E}\left[ b_4f_4 \Phi_p^4  a_3f_3\Phi_p^3 \mathbb{I}[X_4 \le X_3] \right]
  +  \mathbb{E}\left[f_4^2 \Phi_p^4 a_3b_3  \Phi_p^3 \mathbb{I}[X_4
    < X_3] \right] \\
  & = \mathbb{E}\left[a_3b_3\Phi_p^3 f_4^2 \Phi_p^4 ( \mathbb{I}[X_3
    < X_4]+ \mathbb{I}[X_4
    < X_3])  \right] - \mathbb{E}\left[ b_4f_4 \Phi_p^4
    a_3f_3\Phi_p^3 ( \mathbb{I}[X_3
    < X_4] + \mathbb{I}[X_4 < X_3]) \right] 
\end{align*}
which leads to 
  \begin{equation}
    \label{eq:73}
    R = \mathbb{E} \left[
      \begin{pmatrix}
        a_3f_4 -a_4f_3 \\
        b_3f_4 -b_4f_3 
      \end{pmatrix}      \begin{pmatrix}
        a_3f_4 -a_4f_3  & 
        b_3f_4 -b_4f_3 
      \end{pmatrix}
      \Phi_p^{34}
    \right]
  \end{equation}
and  the claim follows.
\end{proof}
\begin{proof}[Proof of Theorem \ref{theo:oklinshepp}]

  The variances can be expressed using \eqref{eq:27} and, using
  successively \eqref{eq:25}, \eqref{eq:reprsntform} and
  \eqref{eq:28}, the covariance matrix can be rewritten by
\begin{align}
&\begin{pmatrix}
      \mathrm{Var}[f(X)] & \mathrm{Cov}[f(X), g(X)] \\
      \mathrm{Cov}[f(X), g(X)] & \mathrm{Var}[g(X)]
    \end{pmatrix}  \nonumber\\
    &=
\E\left[ 
\begin{pmatrix}
f(X_2)-f(X_1) \\
g(X_2)-g(X_1)
\end{pmatrix} 
\begin{pmatrix}
f(X_2)-f(X_1) &
g(X_2)-g(X_1)
\end{pmatrix} \mathbb{I}_{X_1<X_2} 
\right]. 
\end{align}
Using, as usual, the representation
\begin{equation}
f(X_2)-f(X_1) = \E\left[
\frac{\Delta^{-\ell}f(X_3)}{\sqrt{\Delta^{-\ell}h(X_3)}} 
\sqrt{\Delta^{-\ell}h(X_3)} \Phi_P^{\ell}(X_1,X_3,X_2)
| X_1,X_2\right]
\end{equation}
the claim follows by applying the Matrix Cauchy-Schwarz inequality
\eqref{eq:75}  to
the vector
$\mathbf{v}(x)= \begin{pmatrix}
  \frac{\Delta^{-\ell}f(x)}{\sqrt{\Delta^{-\ell}h(x)}} &
  \frac{\Delta^{-\ell}g(x)}{\sqrt{\Delta^{-\ell}h(x)}}
\end{pmatrix}'$ and the function $f(x) = \sqrt{\Delta^{-\ell}h(x)}$.
\end{proof}

\begin{exm}[Normal bounds] 
Direct application of \eqref{eq:72} show that, by taking $h(x)=x$, if $X \sim \mathcal{N}(0, 1)$ then 
\begin{align*}
 & \mathrm{Var}[f(X)] \mathrm{Var}[g(X)] - \{  \mathrm{Cov}[f(X), g(X)] \}^2  
 & \le    \mathbb{E} \left[\big(f'(X)\big)^2 \right]
     \mathbb{E} \left[ \big(g'(X)\big)^2 \right] - \mathbb{E} \left[ f'(X)g'(X)  \right]^2
  \end{align*}
A direct application of \eqref{rem:2} gives us the first order upper bound 
\begin{equation}
\mbox{Var}[g(X)] \le \mathbb{E}[g'(X)^2]. \nonumber
\end{equation}
which was already obtained in \eqref{eq:70} and in Example \ref{ex:normbounds}.
\end{exm}

\begin{exm}[Poisson bounds] 
Direct application of \eqref{eq:72} show that, by taking $h(x)=x$, if $X \sim \mathrm{Pois}(\lambda)$ then 
$
\mathrm{Var}[f(X)] \mathrm{Var}[g(X)] - \{  \mathrm{Cov}[f(X), g(X)] \}^2  
$
  is bounded by
\begin{align*}
    \mathbb{E}
     \left[\big(\Delta^{-}f(X)\big)^2
     X \right]
     \mathbb{E} \left[
     \big(\Delta^{-}g(X)\big)^2 X
     \right]  -
     \mathbb{E} \left[ 
\Delta^{-}f(X)\Delta^{-}g(X) X
     \right]^2
  \end{align*}
  and by
  \begin{align*}
    \lambda^2 \mathbb{E}
     \left[\big(\Delta^{+}f(X)\big)^2
     \right]
     \mathbb{E} \left[
     \big(\Delta^{+}g(X)\big)^2 
     \right]  - \lambda^2
     \mathbb{E} \left[ 
\Delta^{+}f(X)\Delta^{+}g(X) \right]^2.  
  \end{align*}
Moreover, by \eqref{rem:2}, we find two first order upper bounds from Example \ref{ex:poiss3}:
\begin{align*}
\mbox{Var}[g(X)] \le \mathbb{E}[X \big(\Delta^{-}g(X)\big)^2] \mbox{ and }
\mbox{Var}[g(X)] \le \lambda \mathbb{E}[ \big(\Delta^{+}g(X)\big)^2].
\end{align*}
\end{exm}

\section{Applications of the representations to estimating Stein
  factors}
\label{sec:new-repr-stein}

Recall that if $\mathcal{A}_p$ is as in \eqref{eq:33} from Definition
\ref{def:stand-canon-oper} and $\mathcal{H}$ is a collection of
$h : \R \to \R$ belonging to $L^1(p)$, the ``$\mathcal{A}_p$ Stein
equation on $\mathcal{H}$'' is the functional equation
\begin{equation}
  \label{eq:43}
 \mathcal{A}_pg(x) = h(x) - \mathbb{E}[h(X)], \quad x \in \mathcal{S}(p)
\end{equation}
whose  solutions are 
\begin{equation}
  \label{eq:8}
  g(x) = \frac{\mathcal{L}_p^{\ell} h(x)}{\mathcal{L}_p^{\ell} 
 \eta(x)}
\end{equation}
with the convention that $g(x) = 0$ for all $x$ outside of
$\mathcal{S}(p)$. Uniform in $\mathcal{H}$ bounds on $g$ and its
derivatives are known as \emph{Stein factors}; in practice it is
useful to have information on $ \| g_h\|, \|\Delta^{\ell} g_h\|$ also
for some specific functions $h$ (particularly $h(x) = x$). We conclude
the paper with an application of Lemma \ref{lem:represnt} towards
understanding properties of solutions \eqref{eq:8}.  The result is
immediate from previous developments.

\begin{prop}\label{prop:applirep}
 
\begin{enumerate}
\item If $h$ is monotone  then $\mathcal{L}_p^{\ell}h(x)$ does not
  change sign.
  \item If, for all $h \in \mathcal{H}$, there exists a monotone
  function $\eta$ such that $|h(x)-h(y)| \le k|\eta(x)-\eta(y)|$ for
  all $x, y$ then the function $g$ defined in \eqref{eq:8} satisfies
  $ |g(x)| \le k$. In particular if $\mathcal{H}$ is the collection of
  Lipschitz functions with constant $1$ then $\| g\|_{\infty} \le 1$.
\item Let $P$ and $\bar{P}=1-P$ be the cdf and survival function of
  $p$, respectively, and define the function
  \begin{equation}
    \label{eq:9}
    R^{\ell}(x) = \frac{P(x-\ell(\ell+1)/2)
\bar{P}(x+\ell(\ell-1)/2)}{p(x)}
  \end{equation} 
  for all $x \in \mathcal{S}(p)$. We always have
\begin{equation}
  \label{eq:47}
  \left| \mathcal{L}_p^{\ell}h(x)  \right| \le 2 \| h \|_{\infty}
  \frac{\mathbb{E}[\chi^{\ell}(X_1, x)] \mathbb{E}[\chi^{-\ell}(x,
    X_2)]}{p(x)} =  2 \| h \|_{\infty} R^{\ell}(x)
\end{equation}
for all $x \in \mathcal{S}(p)$. 
\end{enumerate}

\end{prop}
\begin{proof}
  Recall representation
  \eqref{eq:reprsntform} which states that
  \begin{equation*}
  -\mathcal{L}_p^{\ell}h(x)  =  \frac{1}{p(x)}\mathbb{E} \left[
       (h(X_2)-h(X_1)) 
  \chi^{\ell}(X_1, x)\chi^{-\ell}(x, X_2)\right]. 
\end{equation*}
\begin{enumerate}
\item If $h(x)$ is monotone the $h(X_2) - h(X_1)$ is of constant sign
  conditionally on  $\chi^{\ell}(X_1, x)\chi^{-\ell}(x, X_2)$.
\item Suppose that the function $\eta$ is strictly decreasing.  By
  definition of $g$ we have, under the stated conditions,
  \begin{align*}
    |g(x)| &  = \left| \frac{\mathbb{E} \left[
       (h(X_2)-h(X_1)) 
  \chi^{\ell}(X_1, x)\chi^{-\ell}(x, X_2)\right]}{\mathbb{E} \left[
       ( \eta(X_2)-\eta(X_1)) 
             \chi^{\ell}(X_1, x)\chi^{-\ell}(x, X_2)\right] } \right|
 \\
    & \le  \frac{\mathbb{E} \left[
       |h(X_2)-h(X_1)| 
  \chi^{\ell}(X_1, x)\chi^{-\ell}(x, X_2)\right]}{\mathbb{E} \left[
       ( \eta(X_2)-\eta(X_1)) 
             \chi^{\ell}(X_1, x)\chi^{-\ell}(x, X_2)\right] } \\
    & \le k.
  \end{align*}
  
\item By definition,
  \begin{align*}
\big|  \mathcal{L}_p^{\ell}h(x)\big| = \mathbb{E} \left[
       |h(X_2)-h(X_1)|
  \chi^{\ell}(X_1, x)\chi^{-\ell}(x, X_2)\right]  \le 2 \|h\|_{\infty}  \mathbb{E} \left[
    \chi^{\ell}(X_1, x)\chi^{-\ell}(x, X_2)\right]
  \end{align*}
  which leads to the conclusion. 
\end{enumerate}
\end{proof}

 \begin{figure}
\begin{subfigure}{.3\textwidth}
  \centering
  \includegraphics[width=.8\linewidth]{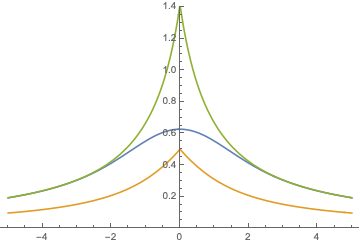}
  \caption{}
  \label{fig:sfig1}
\end{subfigure}%
\begin{subfigure}{.3\textwidth}
  \centering
  \includegraphics[width=.8\linewidth]{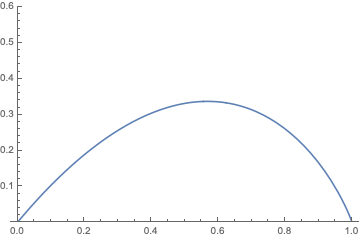}
  \caption{}
  \label{fig:sfig2}
\end{subfigure}
\begin{subfigure}{.3\textwidth}
  \centering
  \includegraphics[width=.8\linewidth]{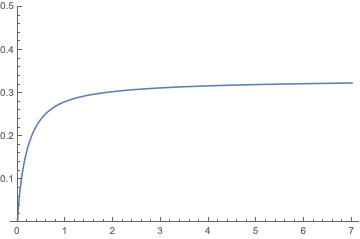}
  \caption{}
  \label{fig:sfig3}
\end{subfigure}

\begin{subfigure}{.3\textwidth}
  \centering
  \includegraphics[width=.8\linewidth]{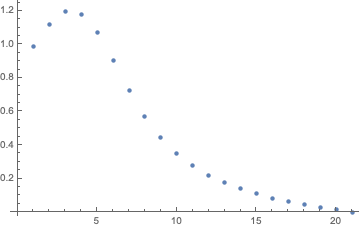}
  \caption{}
  \label{fig:sfigbin}
\end{subfigure}%
\begin{subfigure}{.3\textwidth}
  \centering
  \includegraphics[width=.8\linewidth]{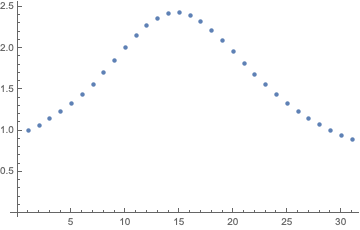}
  \caption{}
  \label{fig:sfigpoi}
\end{subfigure}
\begin{subfigure}{.3\textwidth}
  \centering
  \includegraphics[width=.8\linewidth]{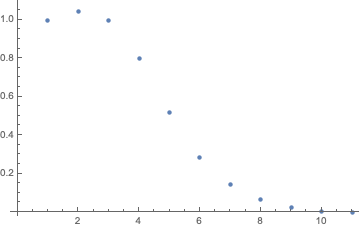}
  \caption{}
  \label{fig:sfighyp}
\end{subfigure}
\caption{{\it The functions $R^{\ell}$ for standard normal distribution
  (Figure~\ref{fig:sfig1}, along with the upper and lower bounds from \eqref{eq:5}), Beta distribution with parameters 8/9 and
1/3 (Figure~\ref{fig:sfig2}) and Gamma distribution with parameters
1/3 and 1/3 (Figure~\ref{fig:sfig3}), binomial distribution with
  parameters $(20, .2)$
  (Figure~\ref{fig:sfigbin}), Poisson distribution with parameter 15  (Figure~\ref{fig:sfigpoi}) and hypergeometric distribution with parameters
10, 10 and 30 (Figure~\ref{fig:sfighyp})}.
\label{fig:figcont}}
\end{figure}

\begin{exm}
  If $p = \phi$ is the standard Gaussian then
  $ R(x) = {\Phi(x)\big(1-\Phi(x)\big)}/{\phi(x)}$ is closely
  related to $r(x)$, the Mill's ratio of the standard normal law.  The
  study of such a function is classical and much is known. For
  instance, we can apply \cite[Theorem 2.3]{baricz2008mills} to get
\begin{equation}\label{eq:5}
\frac{1}{\sqrt{x^2+4}+x}\le \frac{1}{2} r(x)  \le   R(x) \le  r(x)
\le \frac{4}{\sqrt{x^2+8}+3x} 
\end{equation}
for all $x \ge 0$.   Moreover, obviously, $R(x) \le R(0) =
1/2\sqrt{\pi/2} \approx 0.626$. If $h$ is bounded then $\|
\mathcal{L}_{p}^{\ell}h \|_{\infty} \le \sqrt{\pi/2}$ as is well
known, see e.g.\ \cite[Theorem 3.3.1]{NP11}
\end{exm}

\section*{Acknowledgements}

The research of YS was partially supported by the Fonds de la
Recherche Scientifique -- FNRS under Grant no F.4539.16.  ME
acknowledges partial funding via a Welcome Grant of the Universit\'e
de Li\`ege. YS also thanks Lihu Xu for organizing the ``Workshop on
Stein's method and related topics'' at University of Macau in December
2018, and where this contribution was first presented. GR and YS also
thank Emilie Clette for fruitful discussions on a preliminary version
of this work. We also thank Benjamin Arras for several pointers to
relevant literature, as well as corrections on the first draft of the
paper.

Finally, concerning the results in Section \ref{sec:new-repr-stein},
we thank Guillaume Mijoule for communicating with us about a technique
he has developped which not only leads to our bound \eqref{eq:47}, but
also to many more higher order ``Stein factors'', see \cite{M19}.

\bibliographystyle{abbrv} 
\bibliography{biblio_ys}

\end{document}